%% file: paper.tex
\documentclass[11pt]{amsart}

\usepackage{amsmath}
\usepackage{amssymb}
\usepackage{amsthm}
\usepackage{mathtools}
\usepackage{enumitem}
\usepackage{microtype}
\usepackage[colorlinks=true,linkcolor=blue,citecolor=blue,urlcolor=blue]{hyperref}
\usepackage[nameinlink,capitalise,noabbrev]{cleveref}

\newtheorem{theorem}{Theorem}[section]
\newtheorem{proposition}[theorem]{Proposition}
\newtheorem{lemma}[theorem]{Lemma}
\newtheorem{corollary}[theorem]{Corollary}

\theoremstyle{definition}
\newtheorem{definition}[theorem]{Definition}

\theoremstyle{remark}
\newtheorem{remark}[theorem]{Remark}

\numberwithin{equation}{section}

\newcommand{\dd}{\mathrm{d}}

\title[Sharp Regularity Thresholds]{Sharp Regularity Thresholds for Viscous Katz--Pavlovi\'c Dyadic Models}
\author{Ziqian Zhang}
\address{School of Science, Westlake University, Hangzhou, Zhejiang, China}
\email{zhangziqian@westlake.edu.cn}
\subjclass[2020]{35Q30, 35B44, 76D03}
\keywords{dyadic model, global regularity, finite-time blow-up, critical dissipation, super-exponential shells}
\date{}

\begin{document}
\begin{abstract}
\input{sections/abstract}
\end{abstract}
\maketitle
\tableofcontents
\input{sections/1-introduction}
\input{sections/2-preliminaries}
\input{sections/3-acceleration}
\input{sections/4-superexp-blow-up}
\input{sections/5-superexp-regularity}
\input{sections/6-classical-regularity}
\input{sections/7-acknowledgments}
\clearpage
\bibliographystyle{plain}
\bibliography{references}
\end{document}

%% file: sections/abstract.tex
We study the unforced viscous Katz--Pavlovi\'c dyadic model with non-negative smooth initial data for two shell settings. For super-exponential scales $N_n=N_0^{b^n}$, with $N_0>1$ and $1<b<2$, we prove that $\alpha=1/(b+2)$ is the sharp dissipation threshold: every such datum produces a globally smooth solution at and above this value, whereas suitable compactly supported data lose regularity in finite time below it. For geometric scales $N_n=N_0\Lambda^n$, global regularity holds when $\alpha\geq1/3$ and $4\alpha-1>(\log2)/(2\log\Lambda)$. In particular, for $\Lambda>2^{3/2}$, the result reaches $\alpha=1/3$ and, together with the known classical blow-up theorem, identifies the sharp threshold in classical shell ratios. 

%% file: sections/1-introduction.tex
\section{Introduction}
\label{sec:introduction}

The dyadic model is a simplified system of ordinary differential equations used to study the energy transfer in the incompressible Navier--Stokes equations. Viscous Katz--Pavlovi\'c dyadic models describe the competition between nonlinear energy transfer and dissipation. The models in Katz and Pavlovi\'c \cite{KatzPavlovic2005} and Friedlander and Pavlovi\'c \cite{FriedlanderPavlovic2004} retain interactions between neighboring frequency shells. 

We investigate how the dissipation parameter affects the global regularity. The method developed here originated in the study of super-exponential shells and was subsequently applied to the classical model with geometric shells. Its main ingredients are a weighted estimate for the time derivatives and a backward propagation principle. Together they yield a sharp threshold for the super-exponential model and reach the critical exponent in the classical model under an explicit condition on the shell ratio.

\subsection{The models and functional setting}
\label{subsec:introduction-models}

We consider the unforced viscous system
\begin{equation}
  \begin{cases}
    \dot u_n+\nu N_n^{2\alpha}u_n
      =N_{n-1}u_{n-1}^2-N_nu_nu_{n+1}, & n\geq0,\\
    u_{-1}(t)=0, & t\geq0,\\
    u_n(0)=h_n, & n\geq0,
  \end{cases}
  \label{eq:dyadic-model}
\end{equation}
where the incoming term at $n=0$ is defined to be zero and $\nu>0$. We study two choices of shell scales:
\[
  \begin{aligned}
    &N_n=N_0^{b^n}, &&N_0>1,\quad 1<b<2
      &&\text{super-exponential shells},\\
    &N_n=N_0\Lambda^n, &&N_0>0,\quad\Lambda>1
      &&\text{geometric (classical) shells}.
  \end{aligned}
\]
The dissipation parameter is $\alpha\geq0$ in the super-exponential case and $\alpha>0$ in the geometric case. Throughout the paper, the initial datum is non-negative. 

For the chosen sequence $N=(N_n)_{n\geq0}$, set
\[
  H_N^s:=\left\{u:\sum_{n\geq0}N_n^{2s}|u_n|^2<\infty\right\},
  \qquad H_N^\infty:=\bigcap_{s>0}H_N^s.
\]
We call an initial datum smooth if it belongs to $H_N^\infty$. A solution is globally smooth if, for every $s>0$ and every finite $T>0$,
\[
  \sup_{0\leq t\leq T}\|u(t)\|_{H_N^s}<\infty.
\]
The solution classes \textit{i.e.} finite-energy and Leray--Hopf solution, together with the local smooth theory, are specified in Section~\ref{sec:preliminaries}.

\subsection{Main results}
\label{subsec:introduction-results}

Our first result gives the complete threshold classification for super-exponential shells.

\begin{theorem}[Super-exponential threshold]
  \label{thm:main-superexp}
  Let $N_n=N_0^{b^n}$, with $N_0>1$, $1<b<2$, and $\nu>0$.
  \begin{enumerate}
    \item If $\alpha\geq1/(b+2)$, every non-negative $h\in H_N^\infty$ generates a unique global non-negative finite-energy solution of \eqref{eq:dyadic-model}, and this solution is globally smooth.
    \item If $0\leq\alpha<1/(b+2)$, there exist non-negative finitely supported initial data for which the maximal smooth existence time is finite.
  \end{enumerate}
  Thus $\alpha_c=1/(b+2)$ is the sharp threshold for global regularity, and the endpoint belongs to the regular regime.
\end{theorem}

Theorem~\ref{thm:superexp-blowup} gives a more precise statement on the blow-up side: for each $s>1/(b+2)$, suitable finitely supported data force every global non-negative finite-energy solution to lose $H_N^s$ regularity in finite time. The data may depend on $s$. The regularity assertion is proved in Theorem~\ref{thm:superexp-regularity}, including the critical endpoint.

Our second result concerns geometric shells.

\begin{theorem}[Geometric-shell regularity]
  \label{thm:main-classical}
  Let $N_n=N_0\Lambda^n$, with $N_0>0$, $\Lambda>1$, and $\nu>0$. Suppose that
  \begin{equation}
    \alpha\geq\frac13,\qquad
    4\alpha-1>\frac{\log2}{2\log\Lambda}.
    \label{eq:introduction-classical-conditions}
  \end{equation}
  Every non-negative $h\in H_N^\infty$ generates a globally smooth solution of \eqref{eq:dyadic-model}, unique among global non-negative Leray--Hopf solutions with the same initial datum. In particular, these conclusions hold for every $\alpha\geq1/3$ whenever $\Lambda>2^{3/2}$.
\end{theorem}

This is Theorem~\ref{thm:classical-regularity} and its immediate corollary. For every fixed $\Lambda>2^{3/2}$, combining it with Cheskidov's blow-up result for $0<\alpha<1/3$ \cite[Theorem~5.3]{Cheskidov2008} identifies $\alpha=1/3$ as the sharp threshold for global regularity of all non-negative smooth data. The super-exponential blow-up result is proved here with a similar approach. The restriction on $\Lambda$ is part of our regularity result; no optimality of this restriction is asserted.

\subsection{Previous work}
\label{subsec:introduction-literature}

The Katz--Pavlovi\'c dyadic model was introduced in \cite{KatzPavlovic2005} and they proved finite-time blow-up for $\alpha<1/4$ in Corollary~4.2.4. Later, Cheskidov \cite{Cheskidov2008} established global energy solutions, global regularity for $\alpha\geq1/2$, and finite-time blow-up for sufficiently large non-negative data when $\alpha<1/3$. Barbato, Morandin, and Romito \cite{BarbatoMorandinRomito2011} developed an invariant-region argument for neighboring components. Their Theorem~A treats non-negative finite-energy data in the range $2<\beta\leq5/2$ and gives smoothness at positive times. The relation of these results to fluid equations and energy cascades is discussed in the survey \cite{CheskidovDaiFriedlander2023}.

The parameters of \cite{BarbatoMorandinRomito2011} are related to ours by
\[
  \beta=1/\alpha,\qquad \lambda=\Lambda^\alpha,
\]
with the amplitude normalization and index shift given in Section~\ref{sec:preliminaries}. Their choice $\lambda=2$ corresponds to $\Lambda=2^\beta$. Consequently, Theorem~\ref{thm:main-classical} covers $5/2<\beta\leq3$ in that normalization, since $\Lambda=2^\beta>2^{3/2}$. It extends the regularity range for smooth initial data to the threshold $\beta=3$. The initial-data classes should be distinguished: the earlier theorem gives positive-time smoothing from non-negative $\ell^2$ data in its range, whereas our result starts from smooth data.

Palasek's recent work on an Obukhov model with super-exponentially separated shells \cite{Palasek2026} provided the motivation for our study of the super-exponential Katz--Pavlovi\'c model. His viscous blow-up construction uses a smooth external force. Although the nonlinear transfer mechanisms in the two models differ, his construction demonstrates that replacing geometric shells by super-exponential shells can fundamentally change the competition between nonlinear transfer and viscosity. The present work starts from the same shell geometry in the Katz--Pavlovi\'c dynamics and identifies the resulting sharp dissipation threshold.

Very recently, OpenAI reported an AI-generated construction of finite-time blow-up for the forced three-dimensional incompressible Navier--Stokes equations \cite{OpenAI2026NavierStokes}. At the time of writing, this claim is new and awaits independent scrutiny by the mathematical community; if correct, it would constitute a major breakthrough. That construction uses a smooth external force. By contrast, the model studied in this paper is unforced, and our aim is to isolate a mechanism arising solely from nonlinear shell-to-shell transfer and viscous dissipation.

\subsection{Proof strategy}
\label{subsec:introduction-strategy}

Let $[0,T_{\max})$ be the maximal smooth existence interval. Differentiation of \eqref{eq:dyadic-model} gives neighboring cross terms whose internal coefficients cancel under the weights $2^{-n}$. After controlling the boundary term, we obtain the identity for
\[
  \mathcal A(t):=\sum_{n\geq0}2^{-n}|\dot u_n(t)|^2
\]
proved in Theorem~\ref{thm:weighted-acceleration}. Non-negativity makes the remaining terms dissipative, and hence
\begin{equation}
  \mathcal A(t)\leq\mathcal A(0),\qquad
  |\dot u_n(t)|\leq2^{n/2}\sqrt{\mathcal A(0)},
  \qquad t<T_{\max}.
  \label{eq:introduction-acceleration-bound}
\end{equation}
The bound only depends on the initial datum and the fixed parameters, even if $T_{\max}<\infty$. The weights $2^{-n}$ also appear in the uniqueness argument of \cite[Proposition~3.2]{BarbatoMorandinRomito2011}. However, we use these weights for cancellation to build the acceleration estimate. 

For super-exponential shells, fix $q=1/(b+2)$ and write $v_n=N_n^qu_n$. The equation becomes
\[
  v_{n-1}^2
  =\nu N_n^{2(\alpha-q)}v_n+v_nv_{n+1}+R_n,
  \qquad R_n=N_n^{-q}\dot u_n.
\]
If $\alpha\geq q$, the damping coefficient is at least $\nu$. Moreover, $2^{n/2}N_n^{-q}\to0$, so \eqref{eq:introduction-acceleration-bound} makes $R_n$ uniformly small at high indices. Our backward propagation lemma \ref{lem:backward-propagation} rules out the existence of a fixed $\varepsilon>0$ and sequences $n_j\to\infty$ and $t_j\in[0,T_{\max})$ such that $v_{n_j}(t_j)\geq\varepsilon$: these values would force $v_k(t_j)\to\infty$ at a fixed lower shell $k$, contradicting the energy bound. It follows that the critical tail tends to zero uniformly in time. $b\in(1,2)$ allow choosing $p\in(b,2)$, then recover every weighted norm after an arbitrarily short waiting time. The damping remains positive at $\alpha=q$, so the argument includes the endpoint directly.

For geometric shells, the weight is $x_n=N_n^{1-2\alpha}u_n$, and the recurrence is
\[
  A_\alpha x_{n-1}^2
  =\nu x_n+B_\alpha x_nx_{n+1}+R_n,
  \qquad R_n=N_n^{1-4\alpha}\dot u_n,
\]
where $A_\alpha=\Lambda^{1-4\alpha}$ and $B_\alpha=\Lambda^{2\alpha-1}$. The condition $B_\alpha\geq A_\alpha$ is equivalent to $\alpha\geq1/3$ as required in lemma \ref{lem:backward-propagation}. In this geometry, $2^{n/2}=N_0^{-\rho}N_n^\rho$, where $\rho=(\log2)/(2\log\Lambda)$. The error therefore tends to zero under the second condition in \eqref{eq:introduction-classical-conditions}. Backward propagation again gives uniform critical-tail decay, and the classical criterion in Proposition~\ref{prop:classical-tail-criterion} yields smoothness and continuation. This accounts for the additional shell-ratio condition in the transfer of the method.

The super-exponential blow-up argument below $1/(b+2)$ uses a separate weighted Lyapunov construction, which is similar to previous techniques on geometric shells in \cite{Cheskidov2008}. The chosen linear functional satisfies a Riccati lower bound for sufficiently large initial data, contradicting its energy-based upper bound. This completes the sharp classification. 

\subsection{Organization of the paper}
\label{subsec:introduction-organization}

Section~\ref{sec:preliminaries} gives the solution theory, and Section~\ref{sec:acceleration} develops the common estimates. Sections~\ref{sec:superexp-blowup} and~\ref{sec:superexp-regularity} establish the super-exponential threshold. Section~\ref{sec:classical-regularity} treats geometric shells and compares the models. Section~\ref{sec:acknowledgments} contains the acknowledgments.

The author used OpenAI Codex to assist with editing and formatting the manuscript. The author reviewed and revised all AI-assisted material and takes full responsibility for the content of this paper.

%% file: sections/2-preliminaries.tex
\section{Preliminaries and Solution Theory}
\label{sec:preliminaries}

We first recall the functional setting and the classical solution theory for geometric shells, following \cite{Cheskidov2008,BarbatoMorandinRomito2011}. We then use the same solution concepts for super-exponential shells and establish the energy estimates, local smoothness, uniqueness, and continuation results needed below.

\subsection{The classical model and functional setting}
\label{subsec:classical-setting}

For the classical model \eqref{eq:dyadic-model}, the shell scales are
\begin{equation}
  N_n=N_0\Lambda^n,\qquad N_0>0,\quad \Lambda>1,
  \label{eq:geometric-shells}
\end{equation}
and $\nu>0$. The incoming term in the equation for $n=0$ is defined to be zero. We use $\alpha>0$ for the classical local theory below. The super-exponential construction will also allow $\alpha=0$. Initial data are always non-negative.

For $s,q\in\mathbb R$, let $H_N^s$ and $X_N^q$ be the spaces of real sequences with finite norms
\[
  \|u\|_{H_N^s}^2:=\sum_{n\geq0}N_n^{2s}|u_n|^2,
  \qquad
  \|u\|_{X_N^q}:=\sup_{n\geq0}N_n^q|u_n|.
\]
In particular, $H_N^0=\ell^2$ is the energy space. We also set
\[
  X_{N,0}^q:=\{u\in X_N^q:N_n^qu_n\longrightarrow0\},
  \qquad
  H_N^\infty:=\bigcap_{s>0}H_N^s.
\]
The space $X_{N,0}^q$ is a closed subspace of $X_N^q$. It is useful for the local construction because the diagonal heat semigroup is strongly continuous on $X_{N,0}^q$, while it is not on $X_N^q$. 

For every $\varepsilon>0$ and $s\in\mathbb R$,
\begin{equation}
  \|u\|_{X_N^s}\leq\|u\|_{H_N^s},\qquad
  \|u\|_{H_N^s}
  \leq\left(\sum_{n\geq0}N_n^{-2\varepsilon}\right)^{1/2}
       \|u\|_{X_N^{s+\varepsilon}}.
  \label{eq:weighted-embeddings}
\end{equation}
The numerical series converges for geometric shells. Thus
$H_N^\infty=\bigcap_{s>0}X_N^s =: X_{N}^{\infty}$. We use the same notation and definitions below when $N$ is super-exponential.

\begin{definition}[Finite-energy and Leray--Hopf solutions]
  \label{def:energy-solutions}
  A finite-energy solution, also called a weak solution, of \eqref{eq:dyadic-model} on an interval $I\subset[0,\infty)$ is an $\ell^2$-valued map $u$ whose coordinates belong to $C^1(I)$ and satisfy the coordinate equations. At included endpoints, derivatives are understood from within $I$; when $0\in I$, the initial condition is satisfied coordinatewise.

  A finite-energy solution on $[0,T)$ is a Leray--Hopf solution if
  \begin{equation}
    \|u(t)\|_{\ell^2}^2
    +2\nu\int_{t_0}^t\|u(\tau)\|_{H_N^\alpha}^2\,\dd\tau
    \leq\|u(t_0)\|_{\ell^2}^2
    \label{eq:energy-inequality}
  \end{equation}
  holds for $t_0=0$ and for almost every $t_0\in(0,T)$, and for every $t\in[t_0,T)$. The definition on $[0,T]$ is analogous. For the non-negative solutions considered here, the inequality will hold for every initial time $t_0$.
\end{definition}

\begin{definition}[Smoothness and finite-time loss of regularity]
  \label{def:smoothness}
  An initial datum is smooth if it belongs to $H_N^\infty$. A finite-energy solution is smooth on $I$ if, for every compact interval $J\subset I$ and every $s>0$,
  \[
    \sup_{t\in J}\|u(t)\|_{H_N^s}<\infty.
  \]
  It is globally smooth, or globally regular, if it is smooth on $[0,\infty)$. A global finite-energy solution loses $H_N^s$ regularity in finite time if, for some $T<\infty$,
  \[
    \sup_{0\leq t\leq T}\|u(t)\|_{H_N^s}=\infty.
  \]
\end{definition}

By \eqref{eq:weighted-embeddings}, smoothness is equivalently expressed by uniform $X_N^s$ bounds for every $s>0$ on compact time intervals. Each coordinate of a finite-energy solution is automatically smooth in time by repeated differentiation of the equations. This coordinatewise property does not imply smoothness in the sense of Definition~\ref{def:smoothness}.

\subsection{Classical solution theory and a regularity criterion}
\label{subsec:classical-theory}

\begin{proposition}[Classical energy solutions]
  \label{prop:classical-energy}
  Let $N$ satisfy \eqref{eq:geometric-shells}, and let $\nu>0$ and $\alpha>0$. For every non-negative $h\in\ell^2$, there exists a global non-negative Galerkin solution of \eqref{eq:dyadic-model}. Every non-negative finite-energy solution satisfies \eqref{eq:energy-inequality} for every pair of times $t_0\leq t$ in its interval of existence. If the solution is smooth on $[t_0,t]$, then the energy inequality is an equality.
\end{proposition}

The existence assertion is the Galerkin construction of \cite[Theorems~4.1--4.2]{Cheskidov2008}, expressed in the normalization \eqref{eq:dyadic-model}. The approximations solve the equations for $0\leq n\leq M$ with $u_{M+1}^{(M)}=0$. Their energy bounds give a subsequence converging locally uniformly in time in every fixed coordinate by Arzel\`a--Ascoli. 

We record the identities underlying positivity and the energy statement, since they also apply to the super-exponential model. For $n\geq1$, the integrating-factor formula is
\begin{align}
  u_n(t)={}&u_n(t_0)
  \exp\left(-\int_{t_0}^t(\nu N_n^{2\alpha}+N_nu_{n+1}(\tau))\,\dd\tau\right)
  \notag\\
  &+N_{n-1}\int_{t_0}^t u_{n-1}(r)^2
  \exp\left(-\int_r^t(\nu N_n^{2\alpha}+N_nu_{n+1}(\tau))\,\dd\tau\right)\,\dd r.
  \label{eq:integrating-factor}
\end{align}
For $n=0$, the second term is omitted. In particular, non-negative initial data remain non-negative. Multiplying the equations by $2u_n$ and summing gives
\begin{equation}
  \begin{aligned}
    \sum_{n=0}^M u_n(t)^2
    &+2\nu\int_{t_0}^t\sum_{n=0}^M N_n^{2\alpha}u_n(\tau)^2\,\dd\tau\\
    &+2\int_{t_0}^t N_Mu_M(\tau)^2u_{M+1}(\tau)\,\dd\tau
      =\sum_{n=0}^M u_n(t_0)^2.
  \end{aligned}
  \label{eq:truncated-energy}
\end{equation}
For non-negative solutions, discarding the boundary flux and passing to the limit proves \eqref{eq:energy-inequality} for every $t_0$. If $u$ is smooth, then, for any $r>1/3$,
\[
  \sup_{\tau\in[t_0,t]}|N_Mu_M(\tau)^2u_{M+1}(\tau)|
  \leq C_r^3\Lambda^{-r}N_M^{1-3r}\longrightarrow0,
\]
which gives the energy equality.

To recall the classical local theory, we specify the conversion to the notation of \cite{BarbatoMorandinRomito2011}. Put
\[
  c_0:=N_0/\Lambda,\qquad
  \lambda:=\Lambda^\alpha,\qquad
  \lambda_k:=\lambda^k,\qquad
  \beta:=1/\alpha,
\]
and set $X_k:=c_0u_{k-1}$ for $k\geq1$, $X_0:=0$, and $\widetilde\nu:=\nu c_0^{2\alpha}$. Then \eqref{eq:dyadic-model} becomes
\begin{equation}
  \dot X_k+\widetilde\nu\lambda_k^2X_k
  =\lambda_{k-1}^{\beta}X_{k-1}^2
   -\lambda_k^\beta X_kX_{k+1},\qquad k\geq1.
  \label{eq:classical-normalization}
\end{equation}
The shift of indices accounts for the convention that the first active shell in the reference is $k=1$. The critical weighted variable in \cite{BarbatoMorandinRomito2011} satisfies
\[
  \lambda_k^{\beta-2}|X_k|
  =c_0^{2\alpha}N_{k-1}^{1-2\alpha}|u_{k-1}|.
\]
Accordingly, define
\begin{equation}
  q_{\mathrm{cl}}(\alpha):=1-2\alpha.
  \label{eq:classical-critical-weight}
\end{equation}

\begin{proposition}[Classical local smoothness and continuation]
  \label{prop:classical-local}
  Let $N$ satisfy \eqref{eq:geometric-shells}, with $\nu>0$ and $\alpha>0$, and let $h\in H_N^\infty$ be non-negative. Every non-negative Galerkin solution with initial datum $h$ agrees locally with a smooth solution. On each compact smooth interval, this solution is unique among non-negative Leray--Hopf solutions with the same initial datum.

  Let $T_{\max}\in(0,\infty]$ denote its maximal smooth existence time. If $T_{\max}<\infty$, then, for every $q>\max\{0,q_{\mathrm{cl}}(\alpha)\}$,
  \begin{equation}
    \limsup_{t\uparrow T_{\max}}\|u(t)\|_{X_N^q}=\infty.
    \label{eq:classical-continuation}
  \end{equation}
\end{proposition}

Under \eqref{eq:classical-normalization}, the construction follows \cite[Remark~3.4]{BarbatoMorandinRomito2011}. We record the estimates that give a uniform local existence time for every fixed geometric ratio $\Lambda>1$. Fix $q>\max\{0,q_{\mathrm{cl}}(\alpha)\}$ and set
\[
  \begin{gathered}
    E_T^q:=C([0,T];X_{N,0}^q),\qquad
    \|u\|_{E_T^q}:=\sup_{0\leq t\leq T}\|u(t)\|_{X_N^q},\\
    \theta:=\max\left\{0,\frac{1-q}{2\alpha}\right\}<1.
  \end{gathered}
\]
Write $B_0(u):=-N_0u_0u_1$ and $B_n(u):=N_{n-1}u_{n-1}^2-N_nu_nu_{n+1}$ for $n\geq1$. The geometric shell relation gives
\[
  \begin{aligned}
    \|B(u)\|_{X_N^{2q-1}}&\leq C\|u\|_{X_N^q}^2,\\
    \|B(u)-B(v)\|_{X_N^{2q-1}}
      &\leq C(\|u\|_{X_N^q}+\|v\|_{X_N^q})\|u-v\|_{X_N^q}.
  \end{aligned}
\]
Moreover, $B$ maps $X_{N,0}^q$ into $X_{N,0}^{2q-1}$, since the weighted products tend to zero. For $(S(t)z)_n:=\exp(-\nu N_n^{2\alpha}t)z_n$ and $0<t\leq1$,
\[
  \|S(t)z\|_{X_N^q}\leq C t^{-\theta}\|z\|_{X_N^{2q-1}}.
\]
Indeed, the required multiplier is $N_n^{1-q}\exp(-\nu N_n^{2\alpha}t)$; it is bounded by $Ct^{-(1-q)/(2\alpha)}$ when $q<1$, and by a constant when $q\geq1$. The semigroup is strongly continuous on $X_{N,0}^q$. Since $\theta<1$, the mapping
\[
  (\mathcal D u)(t):=\int_0^t S(t-r)B(u(r))\,\dd r
\]
belongs to $E_T^q$ for $u\in E_T^q$ and $0<T\leq1$, and satisfies
\[
  \begin{aligned}
    \|\mathcal D u\|_{E_T^q}
      &\leq C T^{1-\theta}\|u\|_{E_T^q}^2,\\
    \|\mathcal D u-\mathcal D v\|_{E_T^q}
      &\leq C T^{1-\theta}(\|u\|_{E_T^q}+\|v\|_{E_T^q})
         \|u-v\|_{E_T^q}.
  \end{aligned}
\]
Here the constants depend only on $q$ and the fixed model parameters. Given $R>0$ and $\|h\|_{X_N^q}\leq R$, the map $u\mapsto S(\cdot)h+\mathcal D u$ is a contraction on the closed ball of radius $2R$ in $E_T^q$ whenever $4CRT^{1-\theta}\leq1/2$. Thus one may take
\[
  T=\min\left\{1,(8CR)^{-1/(1-\theta)}\right\}.
\]
The same estimates hold uniformly for the Galerkin truncations. The mild solution satisfies the coordinate equations; since $q>0$, it is finite-energy, and \eqref{eq:integrating-factor} preserves non-negativity.

The regularity criterion below makes the locally constructed solution smooth for smooth initial data. The comparison argument of \cite[Proposition~3.2]{BarbatoMorandinRomito2011} gives uniqueness in class of non-negative Leray--Hopf solution. This identifies the local solution with the Galerkin solution. Finally, if \eqref{eq:classical-continuation} failed, its $X_N^q$ norm would stay bounded up to $T_{\max}$. Restarting the local construction at times approaching $T_{\max}$, with a common positive lifespan, and applying the following criterion would extend smoothness beyond $T_{\max}$.

\begin{proposition}[Classical critical-tail criterion]
  \label{prop:classical-tail-criterion}
  Let $N$ satisfy \eqref{eq:geometric-shells}, and let $u$ be a non-negative finite-energy solution on $[0,T]$ with $h\in H_N^\infty$. If
  \begin{equation}
    \lim_{n\to\infty}\sup_{0\leq t\leq T}
    N_n^{q_{\mathrm{cl}}(\alpha)}u_n(t)=0,
    \label{eq:classical-critical-tail}
  \end{equation}
  then $u$ is smooth on $[0,T]$. In particular, this holds if
  $\sup_{0\leq t\leq T}\|u(t)\|_{X_N^q}<\infty$ for some $q>q_{\mathrm{cl}}(\alpha)$.
\end{proposition}

This is the non-negative case of \cite[Proposition~3.3]{BarbatoMorandinRomito2011}, in our normalization. We include the short argument to make the dependence on the geometric ratio explicit.

\begin{proof}
  Fix $s>0$ and write
  \[
    A_n:=\sup_{0\leq t\leq T}N_n^{q_{\mathrm{cl}}(\alpha)}u_n(t),
    \qquad G_n:=\sup_{0\leq t\leq T}N_n^su_n(t).
  \]
  Each $G_n$ is finite by coordinate continuity. Positivity and \eqref{eq:integrating-factor} imply, for $n\geq1$,
  \[
    u_n(t)\leq h_n\exp(-\nu N_n^{2\alpha}t)
      +N_{n-1}\int_0^t\exp(-\nu N_n^{2\alpha}(t-r))u_{n-1}(r)^2\,\dd r.
  \]
  Since $1-q_{\mathrm{cl}}(\alpha)=2\alpha$, integration of the exponential kernel yields
  \[
    G_n\leq\|h\|_{X_N^s}
      +\frac{\Lambda^{s-2\alpha}}{\nu}A_{n-1}G_{n-1}.
  \]
  Choose $n_0\geq1$ so that the coefficient of $G_{n-1}$ is at most $1/2$ for $n\geq n_0$. Induction gives
  \[
    \sup_{n\geq n_0-1}G_n
    \leq\max\{G_{n_0-1},2\|h\|_{X_N^s}\}<\infty.
  \]
  The remaining coordinates are bounded as well. Since $s>0$ is arbitrary, \eqref{eq:weighted-embeddings} proves smoothness. A uniform $X_N^q$ bound with $q>q_{\mathrm{cl}}(\alpha)$ implies \eqref{eq:classical-critical-tail} by multiplication by $N_n^{q_{\mathrm{cl}}(\alpha)-q}$.
\end{proof}

\subsection{Super-exponential shells and energy solutions}
\label{subsec:superexp-energy}

We now consider \eqref{eq:dyadic-model} with
\begin{equation}
  N_n=N_0^{b^n},\qquad N_0>1,\quad b>1,
  \label{eq:superexp-shells}
\end{equation}
so that $N_{n+1}=N_n^b$. Throughout this subsection, $\nu>0$ and $\alpha\geq0$. The restriction $b<2$ is not needed for the energy construction.

We use the spaces and solution concepts of Subsection~\ref{subsec:classical-setting} with the scales \eqref{eq:superexp-shells}. For every $c>0$,
\begin{equation}
  \sum_{n\geq0}N_n^{-c}
  =\sum_{n\geq0}\exp(-cb^n\log N_0)<\infty.
  \label{eq:superexp-summability}
\end{equation}
Indeed, the ratio of consecutive terms tends to zero. Hence \eqref{eq:weighted-embeddings} remains valid, and smoothness is again equivalent to bounds in every $X_N^s$ on compact time intervals.

\begin{proposition}[Super-exponential energy solutions]
  \label{prop:superexp-energy}
  Let $N$ satisfy \eqref{eq:superexp-shells}, with $b>1$, $\nu>0$, and $\alpha\geq0$. For every non-negative $h\in\ell^2$, there exists a global non-negative finite-energy solution of \eqref{eq:dyadic-model}, obtained by Galerkin approximation. It satisfies
  \begin{equation}
    \|u(t)\|_{\ell^2}^2
    +2\nu\int_0^t\|u(\tau)\|_{H_N^\alpha}^2\,\dd\tau
    \leq\|h\|_{\ell^2}^2,\qquad t\geq0.
    \label{eq:superexp-galerkin-bound}
  \end{equation}
\end{proposition}

\begin{proof}
  Let $P_M$ project onto coordinates $0,\ldots,M$. Solve the finite system with initial datum $P_Mh$ and upper boundary $u_{M+1}^{(M)}=0$. Cancellation gives
  \[
    \sum_{n=0}^M|u_n^{(M)}(t)|^2
    +2\nu\int_0^t\sum_{n=0}^M N_n^{2\alpha}|u_n^{(M)}(\tau)|^2\,\dd\tau
    =\|P_Mh\|_{\ell^2}^2.
  \]
  This prevents finite-time escape in the finite-dimensional system, so each approximation is global. Formula~\eqref{eq:integrating-factor} also applies to the finite system and preserves non-negativity.

  Put $R:=\|h\|_{\ell^2}$. Every coordinate is bounded by $R$, and, for fixed $n\geq1$ and $M\geq n+1$,
  \[
    |\dot u_n^{(M)}(t)|
    \leq\nu N_n^{2\alpha}R+(N_{n-1}+N_n)R^2.
  \]
  For $n=0$, omit the $N_{n-1}$ term. These bounds are independent of $M$ and $t$. Arzel\`a--Ascoli and a diagonal extraction therefore give integers $M_j\to\infty$ and functions $u_n$ such that $u_n^{(M_j)}\to u_n$ uniformly on every compact time interval, for every fixed $n$.

  The integral equation for coordinate $n$ involves only coordinates $n-1,n,n+1$. Their uniform convergence permits passage to the limit in that equation. Consequently, $u_n\in C^1([0,\infty))$, $u_n(0)=h_n$, and \eqref{eq:dyadic-model} holds coordinatewise. Non-negativity passes to the limit.

  For a fixed $L$, retain only coordinates $0,\ldots,L$ in the finite energy identity and let $j\to\infty$. Uniform coordinate convergence gives
  \[
    \sum_{n=0}^L u_n(t)^2
    +2\nu\int_0^t\sum_{n=0}^L N_n^{2\alpha}u_n(\tau)^2\,\dd\tau
    \leq R^2.
  \]
  Letting $L\to\infty$ and using monotone convergence proves \eqref{eq:superexp-galerkin-bound}. In particular, $u(t)\in\ell^2$ for every $t\geq0$.
\end{proof}

This is the classical Galerkin strategy of \cite{Cheskidov2008}; the proof records why it remains valid when the shell ratios are unbounded.

\begin{proposition}[Super-exponential energy balance]
  \label{prop:superexp-energy-balance}
  Every non-negative finite-energy solution with scales \eqref{eq:superexp-shells} satisfies \eqref{eq:energy-inequality} for every $t_0\leq t$ in its interval of existence. Moreover, the limit
  \[
    \mathfrak F(t_0,t)
    :=2\lim_{M\to\infty}\int_{t_0}^t N_Mu_M(\tau)^2u_{M+1}(\tau)\,\dd\tau
  \]
  exists and is non-negative, and
  \begin{equation}
    \|u(t)\|_{\ell^2}^2
    +2\nu\int_{t_0}^t\|u(\tau)\|_{H_N^\alpha}^2\,\dd\tau
    +\mathfrak F(t_0,t)=\|u(t_0)\|_{\ell^2}^2.
    \label{eq:superexp-flux-balance}
  \end{equation}
  If $u$ is smooth on $[t_0,t]$, then $\mathfrak F(t_0,t)=0$.
\end{proposition}

\begin{proof}
  The finite identity \eqref{eq:truncated-energy} is unchanged by the choice of shells. Its boundary flux is non-negative. Discarding that term and letting $M\to\infty$ proves the energy inequality by monotone convergence. The dissipation integral is therefore finite. Solving \eqref{eq:truncated-energy} for the integrated boundary flux expresses it as a difference of convergent quantities, proving existence of the limit and \eqref{eq:superexp-flux-balance}; monotonicity of the flux in $M$ is not required.

  If $u$ is smooth, choose $r>1/(b+2)$. Since $N_{M+1}=N_M^b$,
  \[
    \sup_{\tau\in[t_0,t]}|N_Mu_M(\tau)^2u_{M+1}(\tau)|
    \leq C_r^3N_M^{1-(b+2)r}\longrightarrow0.
  \]
  Hence the integrated flux vanishes.
\end{proof}

\subsection{Local existence and propagation of smoothness}
\label{subsec:superexp-local}

For the remainder of this section, assume \eqref{eq:superexp-shells} with $1<b<2$, $\nu>0$, and $\alpha\geq0$. We adapt the weighted-space construction underlying \cite[Remark~3.4]{BarbatoMorandinRomito2011}. The estimates below use a sufficiently strong weight and do not require a positive dissipation exponent. Similar to Proposition \ref{prop:classical-local}, we have following local existence and continuation results for super-exponential shells.

\begin{proposition}[Local construction in a weighted space]
  \label{prop:superexp-local-construction}
  Let $\sigma\geq1/(2-b)$. For every $h\in X_{N,0}^\sigma$, there exists a unique local mild solution of \eqref{eq:dyadic-model} in $C([0,T];X_{N,0}^\sigma)$. The existence time and the solution norm can be controlled by a bound on $\|h\|_{X_N^\sigma}$ and the fixed parameters, uniformly for the Galerkin truncations. Non-negative initial data give a non-negative solution.
\end{proposition}

\begin{proof}
  Set $B_0(u):=-N_0u_0u_1$ and
  $B_n(u):=N_{n-1}u_{n-1}^2-N_nu_nu_{n+1}$ for $n\geq1$. The incoming term satisfies
  \[
    N_n^\sigma N_{n-1}|u_{n-1}|^2
    \leq N_{n-1}^{1-(2-b)\sigma}\|u\|_{X_N^\sigma}^2,
  \]
  while the outgoing term satisfies
  \[
    N_n^\sigma N_n|u_nu_{n+1}|
    \leq N_n^{1-b\sigma}\|u\|_{X_N^\sigma}^2.
  \]
  Both coefficients are bounded. At $\sigma=1/(2-b)$, the incoming weighted term still tends to zero for $u\in X_{N,0}^\sigma$, since it is the square of a vanishing weighted coordinate. The outgoing coefficient has a strictly negative exponent. This implies $B(u)\in X_{N,0}^\sigma$ and first inequality. Factoring differences of squares and products gives second inequality.
  \begin{align}
    \|B(u)\|_{X_N^\sigma}
      &\leq C_\sigma\|u\|_{X_N^\sigma}^2,\notag\\
    \|B(u)-B(v)\|_{X_N^\sigma}
      &\leq C_\sigma(\|u\|_{X_N^\sigma}+\|v\|_{X_N^\sigma})
           \|u-v\|_{X_N^\sigma}.
    \label{eq:superexp-nonlinear-bounds}
  \end{align}
  Thus $B$ is locally Lipschitz from $X_{N,0}^\sigma$ into itself.

  The semigroup $(S(t)u)_n:=\exp(-\nu N_n^{2\alpha}t)u_n$ is a contraction on $X_{N,0}^\sigma$. Its strong continuity follows by splitting into a uniformly small weighted tail and a finite block. On $C([0,T];X_{N,0}^\sigma)$, consider
  \[
    (\Phi u)(t):=S(t)h+\int_0^t S(t-r)B(u(r))\,\dd r.
  \]
  If $R:=\|h\|_{X_N^\sigma}>0$, then on the ball of radius $2R$ the map has norm at most $R+4C_\sigma TR^2$ and Lipschitz constant at most $4C_\sigma TR$. Taking $T\leq(8C_\sigma R)^{-1}$ gives a contraction. For $R=0$, the zero solution is obtained, and uniqueness follows from the same local Lipschitz estimate. In general, let $u$ and $v$ be two mild solutions on a common compact interval $[t_0,t_1]$, with values in $X_{N,0}^\sigma$. Writing their mild equations from time $t_0$ and using that $S(t)$ is a contraction, set
  \[
    d(t):=\|u(t)-v(t)\|_{X_N^\sigma},\qquad
    L(r):=\|u(r)\|_{X_N^\sigma}+\|v(r)\|_{X_N^\sigma}.
  \]
  The local Lipschitz bound \eqref{eq:superexp-nonlinear-bounds} gives
  \[
    d(t)\leq d(t_0)+C_\sigma\int_{t_0}^t L(r)d(r)\,\dd r,
    \qquad t_0\leq t\leq t_1.
  \]
  Since $L$ is continuous, Gronwall's inequality yields
  \[
    d(t)\leq d(t_0)\exp\!\left(C_\sigma\int_{t_0}^t L(r)\,\dd r\right).
  \]
  In particular, if the two solutions have the same value at time $t_0$, then $d(t_0)=0$ and hence $u=v$ throughout $[t_0,t_1]$.

  Since Galerkin projection $P_M$ is a contraction, $P_MB(P_Mu)$ satisfies the same estimates. The same construction therefore gives the uniform Galerkin bounds. Finally, each coordinate of the mild solution satisfies the differential equation. Since $\sigma>0$, \eqref{eq:superexp-summability} implies that the solution takes values in $\ell^2$. Formula~\eqref{eq:integrating-factor} then proves preservation of non-negativity.
\end{proof}

Similar to Proposition \ref{prop:classical-tail-criterion}, we have following smoothness criterion for super-exponential shells.
\begin{lemma}[Propagation of higher weights]
  \label{lem:superexp-higher-weights}
  Let $u$ be a non-negative finite-energy solution on $[0,T]$ with $h\in H_N^\infty$. If, for some $r>1/(2-b)$,
  \[
    M_r:=\sup_{0\leq t\leq T}\|u(t)\|_{X_N^r}<\infty,
  \]
  then $u$ is smooth on $[0,T]$.
\end{lemma}

\begin{proof}
  Dropping the damping factors in \eqref{eq:integrating-factor} gives, for $n\geq1$,
  \[
    0\leq u_n(t)
    \leq h_n+N_{n-1}\int_0^t u_{n-1}(\tau)^2\,\dd\tau
    \leq h_n+TM_r^2N_{n-1}^{1-2r}.
  \]
  Put $G(r):=(2r-1)/b$. Since $N_{n-1}^{1-2r}=N_n^{-G(r)}$, this gives a uniform $X_N^{G(r)}$ bound. The first coordinate satisfies $0\leq u_0(t)\leq h_0$. Iterating from $r_0=r$ yields
  \[
    r_j=\frac1{2-b}
       +\left(\frac2b\right)^j\left(r-\frac1{2-b}\right)
       \longrightarrow\infty.
  \]
  Every step uses the same interval $[0,T]$ and the corresponding finite initial weighted norm. Consequently, all $X_N^s$ norms are uniformly bounded, and \eqref{eq:weighted-embeddings} proves smoothness.
\end{proof}

\begin{corollary}[Local smooth solutions]
  \label{cor:superexp-local-smooth}
  Every non-negative $h\in H_N^\infty$ generates a local non-negative smooth solution. It is unique within the class of non-negative smooth solutions on a common interval.
\end{corollary}

\begin{proof}
  Fix $\sigma>1/(2-b)$. Proposition~\ref{prop:superexp-local-construction} gives a local solution, which is smooth by Lemma~\ref{lem:superexp-higher-weights}. A smooth solution belongs to $C([0,T];X_{N,0}^\sigma)$ for every fixed $\sigma$ on a compact smooth interval: a higher weighted bound controls the tails uniformly, and coordinate continuity controls the finite block. It also satisfies the mild equation. Uniqueness therefore follows from the local construction.
\end{proof}

\subsection{Weak--strong uniqueness and continuation}
\label{subsec:superexp-continuation}

The preceding construction gives uniqueness in a weighted mild-solution class. To identify this local solution with the global energy solutions, we establish uniqueness in comparison with a smooth solution.

\begin{proposition}[Weak--strong uniqueness]
  \label{prop:superexp-weak-strong}
  Let $u$ be a non-negative smooth solution on $[0,T]$. Every non-negative finite-energy solution $v$ on $[0,T]$ with $v(0)=u(0)$ agrees with $u$ there.
\end{proposition}

\begin{proof}
  Both solutions satisfy the energy inequality, and $u$ satisfies the energy equality. In particular, $v$ is uniformly bounded in $\ell^2$. Differentiating the finite cross product, shifting the incoming sums, and passing to the limit gives
  \begin{align}
    \langle u(t),v(t)\rangle-\langle u(0),v(0)\rangle
    ={}&-2\nu\int_0^t\sum_{n\geq0}N_n^{2\alpha}u_nv_n\,\dd\tau
    \notag\\
    &+\int_0^t\mathcal C(\tau)\,\dd\tau,
    \label{eq:cross-product}
  \end{align}
  where
  \[
    \mathcal C:=\sum_{n\geq0}N_n
    \bigl(v_n^2u_{n+1}-v_nv_{n+1}u_n
          +u_n^2v_{n+1}-u_nu_{n+1}v_n\bigr).
  \]
  To justify the limit, the two boundary corrections are
  $N_Mv_M^2u_{M+1}$ and $N_Mu_M^2v_{M+1}$. They vanish uniformly in time because $v$ is bounded in $\ell^2$ and $u$ has arbitrary weighted decay. The same bounds give absolute and uniform convergence of the nonlinear sums. For the viscous sum, Cauchy--Schwarz and a uniform bound on a higher weighted norm of $u$ control the tails uniformly. These facts justify integration and passage to the limit in the finite cross-product identity.

  Set $d:=v-u$. Combining \eqref{eq:cross-product} with the energy equality for $u$ and inequality for $v$, and using $d(0)=0$, yields
  \[
    \|d(t)\|_{\ell^2}^2
    +2\nu\int_0^t\|d(\tau)\|_{H_N^\alpha}^2\,\dd\tau
    \leq-2\int_0^t\mathcal C(\tau)\,\dd\tau.
  \]
  Substitution of $v=u+d$ gives
  \[
    \mathcal C=\sum_{n\geq0}N_n
      \bigl(u_{n+1}d_n^2-u_nd_nd_{n+1}\bigr).
  \]
  The first summand is non-negative. With
  $K:=\sup_{0\leq t\leq T}\|u(t)\|_{X_N^1}<\infty$, we obtain
  \[
    \|d(t)\|_{\ell^2}^2
    \leq2K\int_0^t\sum_{n\geq0}|d_nd_{n+1}|\,\dd\tau
    \leq2K\int_0^t\|d(\tau)\|_{\ell^2}^2\,\dd\tau.
  \]
  Gronwall's inequality proves $d=0$.
\end{proof}

\begin{proposition}[Super-exponential continuation criterion]
  \label{prop:superexp-continuation}
  Let $1<b<2$, $\nu>0$, $\alpha\geq0$, and let $h\in H_N^\infty$ be non-negative. Every non-negative Galerkin solution with initial datum $h$ agrees locally with the unique smooth solution. Let $T_{\max}\in(0,\infty]$ denote its maximal smooth existence time. If $T_{\max}<\infty$, then, for every $\sigma>1/(2-b)$,
  \begin{equation}
    \limsup_{t\uparrow T_{\max}}\|u(t)\|_{X_N^\sigma}=\infty.
    \label{eq:superexp-continuation}
  \end{equation}
  On each compact smooth interval, the solution is unique among non-negative finite-energy solutions with the same initial datum.
\end{proposition}

\begin{proof}
  Corollary~\ref{cor:superexp-local-smooth} supplies a local smooth solution, and Proposition~\ref{prop:superexp-weak-strong} identifies it with every non-negative Galerkin solution. Repeating the construction and using uniqueness on overlaps defines the maximal smooth interval.

  Suppose $T_{\max}<\infty$ and \eqref{eq:superexp-continuation} fails for some $\sigma>1/(2-b)$. Smoothness on compact subintervals then gives
  \[
    M:=\sup_{0\leq t<T_{\max}}\|u(t)\|_{X_N^\sigma}<\infty.
  \]
  At each $t_*<T_{\max}$, the datum $u(t_*)$ is smooth and belongs to $X_{N,0}^\sigma$. Proposition~\ref{prop:superexp-local-construction} gives a restarted solution on $[t_*,t_*+\delta]$, where $\delta>0$ depends only on $M$ and the fixed parameters. Lemma~\ref{lem:superexp-higher-weights} makes it smooth on that interval. Uniqueness identifies it with the previous solution on the overlap. Taking $t_*$ sufficiently close to $T_{\max}$ extends smoothness beyond $T_{\max}$, a contradiction. The last assertion follows from Proposition~\ref{prop:superexp-weak-strong}.
\end{proof}

The exponent $1/(2-b)$ is sufficient for this local construction and continuation argument; no optimality is asserted. 

%% file: sections/3-acceleration.tex
\section{Weighted Acceleration and Backward Propagation}
\label{sec:acceleration}

This section develops the estimates shared in two different shell settings. Differentiating the equations and summing with weights $2^{-n}$ gives a quadratic identity for the time derivatives. The resulting bound makes suitably rescaled time derivatives uniformly small at high shells. An abstract backward propagation lemma then converts a recurrence with a vanishing error into uniform decay of its non-negative solutions.

Let $u$ be a global non-negative Galerkin solution of \eqref{eq:dyadic-model} with smooth initial datum $h$, and let $[0,T_{\max})$ be its maximal smooth existence interval. We use either the geometric shells \eqref{eq:geometric-shells}, with $\alpha>0$, or the super-exponential shells \eqref{eq:superexp-shells}, with $1<b<2$ and $\alpha\geq0$; in both cases $\nu>0$. The local smooth theory and the identification with the Galerkin solution are supplied by Propositions~\ref{prop:classical-local} and \ref{prop:superexp-continuation}, respectively. All time-dependent estimates below are established within the smooth existence interval, without assuming that $T_{\max}$ is infinite.

\subsection{The finite-block acceleration identity}
\label{subsec:finite-acceleration}

Set
\begin{equation}
  p_n:=\dot u_n,\qquad p_{-1}:=0,\qquad
  \mathcal A_M(t):=\sum_{n=0}^M2^{-n}p_n(t)^2.
  \label{eq:finite-acceleration-functional}
\end{equation}
The coordinate equations imply that each $u_n$ is smooth in time, so the following finite-dimensional calculation is justified directly.

\begin{lemma}[Finite-block acceleration identity]
  \label{lem:finite-acceleration}
  For every integer $M\geq0$ and every $0\leq t<T_{\max}$,
  \begin{equation}
    \frac12\mathcal A_M'
    +\sum_{n=0}^M2^{-n}
      \bigl(\nu N_n^{2\alpha}+N_nu_{n+1}\bigr)p_n^2
    =-2^{-M}N_Mu_Mp_Mp_{M+1}.
    \label{eq:finite-acceleration-identity}
  \end{equation}
\end{lemma}

\begin{proof}
  Differentiation of \eqref{eq:dyadic-model} gives
  \begin{equation}
    \dot p_n
    =-\bigl(\nu N_n^{2\alpha}+N_nu_{n+1}\bigr)p_n
      +2N_{n-1}u_{n-1}p_{n-1}-N_nu_np_{n+1},
    \label{eq:differentiated-dyadic-model}
  \end{equation}
  where the incoming term is zero when $n=0$. Multiply by $w_np_n$ and sum from $0$ to $M$. Shifting the incoming sum gives
  \begin{align*}
    \frac12\frac{\dd}{\dd t}\sum_{n=0}^Mw_np_n^2
    &+\sum_{n=0}^Mw_n\bigl(\nu N_n^{2\alpha}+N_nu_{n+1}\bigr)p_n^2\\
    &=\sum_{n=0}^{M-1}(2w_{n+1}-w_n)N_nu_np_np_{n+1}
      -w_MN_Mu_Mp_Mp_{M+1}.
  \end{align*}
  The internal sum is empty if $M=0$. With $w_n=2^{-n}$, all its coefficients vanish, proving \eqref{eq:finite-acceleration-identity}.
\end{proof}

  The cancellation condition $2w_{n+1}=w_n$ determines the positive weights uniquely up to a multiplicative constant and does not use any relation between consecutive shell scales. Non-negativity is needed when interpreting the remaining terms on the left as dissipation. Passing to infinitely many shells also requires control of the boundary term, which we verify for both configurations next.

\subsection{The weighted acceleration estimate}
\label{subsec:weighted-acceleration}

\begin{lemma}[Time-derivative decay and disappearance of the boundary term]
  \label{lem:acceleration-boundary}
  For every compact interval $J\subset[0,T_{\max})$ and every $s>0$, there is a constant $C_{J,s}<\infty$ such that
  \begin{equation}
    \sup_{t\in J}\bigl(|u_n(t)|+|p_n(t)|\bigr)
    \leq C_{J,s}N_n^{-s},\qquad n\geq0.
    \label{eq:time-derivative-decay}
  \end{equation}
  Consequently, the series
  \begin{equation}
    \mathcal A(t):=\sum_{n\geq0}2^{-n}|\dot u_n(t)|^2
    \label{eq:acceleration-functional}
  \end{equation}
  converges uniformly on $J$, and
  \begin{equation}
    \lim_{M\to\infty}\sup_{t\in J}
    \bigl|2^{-M}N_Mu_M(t)p_M(t)p_{M+1}(t)\bigr|=0.
    \label{eq:acceleration-boundary-decay}
  \end{equation}
  In particular, $\mathcal A(0)<\infty$ and $\mathcal A(t)<\infty$ for every $t<T_{\max}$.
\end{lemma}

\begin{proof}
  Smoothness gives a uniform $X_N^r$ bound on $J$ for every $r>0$. In the geometric case, substitution into \eqref{eq:dyadic-model} and use of $N_{n+1}=\Lambda N_n$ yield, for $n\geq1$,
  \[
    \sup_{t\in J}|p_n(t)|
    \leq C_{J,r}\bigl(N_n^{2\alpha-r}+N_n^{1-2r}\bigr).
  \]
  Taking $r\geq\max\{s+2\alpha,(s+1)/2\}$ gives the required decay of $p_n$.

  In the super-exponential case, $N_{n-1}=N_n^{1/b}$ and $N_{n+1}=N_n^b$, so
  \[
    \sup_{t\in J}|p_n(t)|
    \leq C_{J,r}\bigl(
       N_n^{2\alpha-r}
       +N_n^{(1-2r)/b}
       +N_n^{1-(b+1)r}\bigr),\qquad n\geq1.
  \]
  It suffices to take
  \[
    r\geq\max\left\{s+2\alpha,\frac{1+bs}{2},\frac{1+s}{b+1}\right\}.
  \]
  In both cases the first coordinate is handled separately, with its incoming term omitted, and the constant can be enlarged to include it. Together with the assumed decay of $u_n$, these estimates prove \eqref{eq:time-derivative-decay}.

  Since $N_n\geq N_0>0$, the bound
  \[
    2^{-n}\sup_{t\in J}|p_n(t)|^2
    \leq C_{J,s}^2\,2^{-n}N_n^{-2s}
  \]
  is summable in $n$. Thus the series defining $\mathcal A$ converges uniformly on $J$. Moreover, $N_{M+1}\geq N_M$ in both configurations, and hence
  \[
    \sup_{t\in J}|2^{-M}N_Mu_Mp_Mp_{M+1}|
    \leq C_{J,s}^3\,2^{-M}N_M^{1-3s}.
  \]
  Choosing $s=1$ proves \eqref{eq:acceleration-boundary-decay}.
\end{proof}

\begin{theorem}[Weighted acceleration identity]
  \label{thm:weighted-acceleration}
  For every $0\leq t_0\leq t<T_{\max}$,
  \begin{equation}
    \begin{aligned}
      \mathcal A(t)
      +2\int_{t_0}^t\sum_{n\geq0}2^{-n}
      \bigl(\nu N_n^{2\alpha}+N_nu_{n+1}(\tau)\bigr)
      |\dot u_n(\tau)|^2\,\dd\tau
      =\mathcal A(t_0).
    \end{aligned}
    \label{eq:weighted-acceleration-identity}
  \end{equation}
  In particular, $\mathcal A$ is non-increasing on $[0,T_{\max})$.
\end{theorem}

\begin{proof}
  The case $t=t_0$ is immediate. For $t_0<t<T_{\max}$, integrate \eqref{eq:finite-acceleration-identity} over $[t_0,t]$ to obtain
  \begin{align*}
    \frac12\mathcal A_M(t)
    &+\int_{t_0}^t\sum_{n=0}^M2^{-n}
      \bigl(\nu N_n^{2\alpha}+N_nu_{n+1}(\tau)\bigr)p_n(\tau)^2\,\dd\tau\\
    &=\frac12\mathcal A_M(t_0)
      -\int_{t_0}^t2^{-M}N_Mu_M(\tau)p_M(\tau)p_{M+1}(\tau)\,\dd\tau.
  \end{align*}
  Lemma~\ref{lem:acceleration-boundary} makes the boundary integral tend to zero and gives convergence of $\mathcal A_M$ at both endpoints. Since $u_{n+1}\geq0$, the integrands on the left are non-negative and increase with $M$. Monotone convergence therefore gives \eqref{eq:weighted-acceleration-identity}. The dissipation integral is non-negative, proving monotonicity.
\end{proof}

The identity is an exact balance law for weighted acceleration. Here $\mathcal A(t)$ is the instantaneous weighted square of the shell accelerations $\dot u_n(t)$. In the integral, $\nu N_n^{2\alpha}|\dot u_n|^2$ is the viscous dissipation of acceleration at shell $n$, while $N_nu_{n+1}|\dot u_n|^2$ is the additional nonlinear damping induced by the coupling to the next shell; the latter is non-negative because $u_{n+1}\geq0$. And the whole integral measures the dissipation accumulated from $t_0$ to $t$. 

The non-negativity of the dissipation in \eqref{eq:weighted-acceleration-identity} is crucial for our regularity proofs. In the picture of energy transfer towards higher shells, this identity shows that the total weighted acceleration energy is dissipated in time. Intuitively, this dissipation counteracts the ``blow-up of acceleration''  at high shells and provides a mechanism for preventing singularity formation. Combined with the shell-dependent estimates below, it is a key ingredient in proving global regularity in the stated parameter regimes. We also seek an analogous mechanism in the Navier--Stokes equations.

\begin{corollary}[Pointwise acceleration estimate]
  \label{cor:pointwise-acceleration}
  For every $n\geq0$ and every $0\leq t<T_{\max}$,
  \begin{equation}
    |\dot u_n(t)|\leq2^{n/2}\sqrt{\mathcal A(0)}.
    \label{eq:pointwise-acceleration}
  \end{equation}
\end{corollary}

\begin{proof}
  Each term satisfies
  $2^{-n}|\dot u_n(t)|^2\leq\mathcal A(t)\leq\mathcal A(0)$.
\end{proof}

\begin{remark}
  The decay estimates used to remove the boundary term are local to compact smooth time intervals, and their constants may depend on the interval. In contrast, the right-hand side of \eqref{eq:pointwise-acceleration} depends only on the initial datum and the fixed parameters. The estimate therefore holds uniformly throughout $[0,T_{\max})$, even when $T_{\max}<\infty$. No acceleration identity at $T_{\max}$ is needed.
\end{remark}

\subsection{An abstract backward propagation lemma}
\label{subsec:backward-propagation}

The next lemma is a statement about non-negative sequences with a parameter. It excludes a large value at a sufficiently high index by propagating lower bounds towards a fixed index. Allowing the linear damping coefficient to vary will be useful for the super-exponential model.

\begin{lemma}[Backward propagation with variable damping]
  \label{lem:backward-propagation}
  Let $A,B,d_*>0$ with $B\geq A$, and let $I$ be a non-empty parameter set. Suppose that $x_n(t)\geq0$ for $n\geq0$, and that real-valued functions $d_n(t)$ and $r_n(t)$ satisfy
  \begin{equation}
    A x_{n-1}(t)^2
    =d_n(t)x_n(t)+B x_n(t)x_{n+1}(t)+r_n(t),
    \qquad n\geq1,\quad t\in I,
    \label{eq:backward-recurrence}
  \end{equation}
  with $d_n(t)\geq d_*$ for every $n\geq1$ and $t\in I$. Assume also that
  \begin{equation}
    \lim_{n\to\infty}\sup_{t\in I}|r_n(t)|=0,
    \qquad
    \sup_{t\in I}x_k(t)<\infty\quad\text{for every fixed }k\geq0.
    \label{eq:backward-hypotheses}
  \end{equation}
  Then
  \begin{equation}
    \lim_{n\to\infty}\sup_{t\in I}x_n(t)=0.
    \label{eq:backward-tail-conclusion}
  \end{equation}
\end{lemma}

\begin{proof}
  Fix $\varepsilon>0$ and put
  \[
    \eta:=\min\left\{\varepsilon,\frac{d_*}{4A}\right\}.
  \]
  Choose an integer $N\geq1$ such that
  \begin{equation}
    \sup_{t\in I}|r_n(t)|\leq\frac{d_*\eta}{4},\qquad \forall n\geq N.
    \label{eq:backward-small-error}
  \end{equation}
  If $n\geq N$ and $x_{n-1}(t)\leq\eta$, then non-negativity and \eqref{eq:backward-recurrence} imply
  \[
    d_*x_n(t)
    \leq d_n(t)x_n(t)
    \leq A\eta^2+|r_n(t)|
    \leq\frac{d_*\eta}{2}.
  \]
  Thus $x_{n-1}(t)\leq\eta$ forces $x_n(t)\leq\eta/2$. By contraposition and iteration, a value $x_m(t)>\eta$ with $m\geq N+1$ forces
  \begin{equation}
    x_j(t)>\eta,\qquad N-1\leq j\leq m.
    \label{eq:backward-positive-block}
  \end{equation}

  Fix such $m,t$ and suppress the parameter. For $N\leq j\leq m-1$, the recurrence and \eqref{eq:backward-small-error} give
  \[
    A x_{j-1}^2
    \geq d_*x_j+B x_jx_{j+1}-\frac{d_*\eta}{4}
    \geq\frac{3d_*}{4}x_j+B x_jx_{j+1}.
  \]
  Since $B\geq A$, it follows that
  \begin{equation}
    x_{j-1}^2\geq x_j(x_{j+1}+\delta),
    \qquad N\leq j\leq m-1,
    \qquad \delta:=\frac{d_*}{2A}>0.
    \label{eq:backward-comparison-inequality}
  \end{equation}

  Define a comparison sequence by
  \begin{equation}
    L_0=L_1=\eta,\qquad
    L_{k+1}:=\sqrt{L_k(L_{k-1}+\delta)},\qquad k\geq1.
    \label{eq:backward-comparison-sequence}
  \end{equation}
  The last two coordinates in \eqref{eq:backward-positive-block} give
  $x_m\geq L_0$ and $x_{m-1}\geq L_1$. If
  $x_{m-k}\geq L_k$ and $x_{m-k+1}\geq L_{k-1}$, then \eqref{eq:backward-comparison-inequality} gives $x_{m-k-1}\geq L_{k+1}$ whenever $1\leq k\leq m-N$. Consequently,
  \begin{equation}
    x_{m-k}(t)\geq L_k,\qquad 0\leq k\leq m-N+1.
    \label{eq:backward-propagated-bound}
  \end{equation}

  We claim that $L_k\to\infty$. Set $\Delta_k:=L_k-L_{k-1}$ for $k\geq1$. The defining recurrence yields
  \[
    \Delta_{k+1}(L_{k+1}+L_k)=L_k(\delta-\Delta_k).
  \]
  Since $\Delta_1=0$, induction gives $0\leq\Delta_k<\delta$. Indeed, under this hypothesis the right-hand side is positive, so $L_{k+1}\geq L_k$ and
  \[
    0\leq\Delta_{k+1}
    \leq\frac{\delta-\Delta_k}{2}<\delta.
  \]
  Thus $(L_k)$ is non-decreasing. If it were bounded, it would converge to a number $\ell\geq\eta>0$, and \eqref{eq:backward-comparison-sequence} would imply
  $\ell^2=\ell(\ell+\delta)$, which is impossible. This proves the claim.

  Let $K:=\sup_{t\in I}x_{N-1}(t)<\infty$. Choose $m_0\geq N+1$ so large that $L_{m-N+1}>K$ for every $m\geq m_0$. If $x_m(t)>\eta$ for any such $m$ and any $t\in I$, then \eqref{eq:backward-propagated-bound} gives
  \[
    x_{N-1}(t)\geq L_{m-N+1}>K,
  \]
  a contradiction. Therefore $\sup_{t\in I}x_m(t)\leq\eta\leq\varepsilon$ for all $m\geq m_0$. Since $\varepsilon>0$ was arbitrary, \eqref{eq:backward-tail-conclusion} follows.
\end{proof}

\begin{remark}
  No topology or measure on $I$, and no regularity with respect to the parameter, is required. The proof uses only the uniform lower bound $d_n(t)\geq d_*$; no upper bound on the damping coefficients is needed. It also allows $B=A$, which is important for the endpoint applications. 
\end{remark}

\subsection{Vanishing defects}
\label{subsec:vanishing-defects}

In the applications, the error $r(t)$ in Lemma \ref{lem:backward-propagation} is a weighted time derivative. We record the direct consequence of Corollary~\ref{cor:pointwise-acceleration} that verifies its uniform decay.

\begin{lemma}[Vanishing rescaled time derivatives]
  \label{lem:vanishing-defect}
  Let $(\omega_n)_{n\geq0}$ be a real sequence such that
  \[
    \lim_{n\to\infty}2^{n/2}|\omega_n|=0.
  \]
  Set $R_n(t):=\omega_n\dot u_n(t)$. Then
  \begin{equation}
    \lim_{n\to\infty}\sup_{0\leq t<T_{\max}}|R_n(t)|=0.
    \label{eq:vanishing-defect}
  \end{equation}
\end{lemma}

\begin{proof}
  The pointwise acceleration estimate gives
  \[
    \sup_{0\leq t<T_{\max}}|R_n(t)|
    \leq\sqrt{\mathcal A(0)}\,2^{n/2}|\omega_n|,
  \]
  whose right-hand side tends to zero.
\end{proof}

The condition in this lemma takes a different form in the two shell configurations. For super-exponential shells and every $\gamma>0$,
\begin{equation}
  2^{n/2}N_n^{-\gamma}
  =\exp\left(\frac n2\log2-\gamma b^n\log N_0\right)
  \longrightarrow0,
  \label{eq:superexp-defect-decay}
\end{equation}
because $b^n/n\to\infty$. For geometric shells, define
\begin{equation}
  \rho:=\frac{\log2}{2\log\Lambda}>0.
  \label{eq:geometric-acceleration-loss}
\end{equation}
Then
\begin{equation}
  2^{n/2}N_n^{-\gamma}
  =N_0^{-\rho}N_n^{\rho-\gamma},
  \label{eq:geometric-defect-decay}
\end{equation}
which tends to zero if and only if $\gamma>\rho$. At $\gamma=\rho$, this numerical factor is constant, so the pointwise acceleration estimate alone does not imply decay of the corresponding defect.

Once a rescaling puts the equations in the form \eqref{eq:backward-recurrence}, these observations provide the small-error hypothesis of Lemma~\ref{lem:backward-propagation}. The fixed-coordinate bounds in lemma \ref{lem:backward-propagation} come from the energy estimate: if $x_k(t)=c_ku_k(t)$ with fixed positive coefficients $c_k$, then
\[
  \sup_{0\leq t<T_{\max}}x_k(t)
  \leq c_k\|h\|_{\ell^2}<\infty
\]
for each fixed $k$. The remaining tasks are to verify the recurrence coefficients and to turn the resulting uniform tail decay into smoothness. These steps depend on the shell configuration and will be carried out in the respective regularity proofs.

%% file: sections/4-superexp-blow-up.tex
\section{Finite-Time Blow-up for Super-Exponential Shells}
\label{sec:superexp-blowup}

We prove finite-time loss of regularity below the critical dissipation exponent by a weighted Lyapunov argument, following the strategy of \cite{Cheskidov2008}. The choice of weights uses the super-exponential separation of the shells. The proof relies on the energy solution theory in Section~\ref{sec:preliminaries} and does not require the acceleration estimate established in Section~\ref{sec:acceleration}. 

\subsection{The blow-up theorem and proof strategy}
\label{subsec:superexp-blowup-statement}

Throughout this section, we consider \eqref{eq:dyadic-model} with
\[
  N_n=N_0^{b^n},\qquad N_0>1,\quad 1<b<2,\quad \nu>0,
\]
and write
\begin{equation}
  \alpha_c:=\frac{1}{b+2}.
  \label{eq:superexp-critical-exponent}
\end{equation}

\begin{theorem}[Finite-time loss of regularity]
  \label{thm:superexp-blowup}
  Let $0\leq\alpha<\alpha_c$. For every $s>\alpha_c$, there exist a non-negative finitely supported initial datum $h$ and a time $T_R<\infty$ such that every global non-negative finite-energy solution of \eqref{eq:dyadic-model} with initial datum $h$ satisfies
  \begin{equation}
    \sup_{0\leq t\leq T_R}\|u(t)\|_{H_N^s}=\infty.
    \label{eq:superexp-loss-of-regularity}
  \end{equation}
  In particular, the corresponding smooth solution has maximal smooth existence time $T_{\max}\leq T_R<\infty$. The initial datum and $T_R$ may depend on $s$ and on the fixed model parameters.
\end{theorem}

The proof uses a linear functional $H$ and a non-negative quadratic functional $P$. Under a uniform $H_N^s$ bound, suitable weights allow the nonlinear outflow and viscous terms in $H'$ to be absorbed into $P$, up to a finite constant. Weighted Cauchy--Schwarz then yields a Riccati lower bound for $H$. For a sufficiently large initial datum supported on one shell, this lower bound contradicts the uniform upper bound on $H$ supplied by the energy inequality. Thus the assumed $H_N^s$ bound cannot exist. All estimates below allow $\alpha=0$.

\subsection{Choice of the Lyapunov weights}
\label{subsec:superexp-lyapunov-weights}

\begin{lemma}[Choice of weights]
  \label{lem:superexp-lyapunov-weights}
  Let $0\leq\alpha<\alpha_c$ and $s>\alpha_c$. There exists $r$ such that
  \begin{equation}
    \max\left\{0,\frac{4\alpha-1}{2-b},\frac{1-2s}{b}\right\}
    <r<\frac{1}{b+2}.
    \label{eq:superexp-lyapunov-weight-range}
  \end{equation}
  Define
  \begin{equation}
    w_n:=N_n^{-r},\qquad
    a_n:=N_nw_{n+1},\qquad
    c_n:=N_nw_n,\qquad
    d_n:=N_n^{2\alpha}w_n.
    \label{eq:superexp-lyapunov-weights}
  \end{equation}
  There is an integer $n_*\geq1$ such that
  \begin{equation}
    \frac{c_n^2}{a_na_{n+1}}\leq\frac14,\qquad n\geq n_*.
    \label{eq:superexp-outflow-weight-bound}
  \end{equation}
  Moreover, $1-br-2s<0$, and the constants
  \begin{equation}
    C_\nu:=\nu^2\sum_{n=n_*}^{\infty}\frac{d_n^2}{a_n},
    \qquad
    S_r:=\sum_{n=n_*}^{\infty}\frac{w_n^2}{a_n}
    \label{eq:superexp-lyapunov-constants}
  \end{equation}
  are finite and strictly positive.
\end{lemma}

\begin{proof}
  Since $1<b<2$, the hypotheses on $\alpha$ and $s$ imply
  \[
    \frac{4\alpha-1}{2-b}<\frac{1}{b+2},
    \qquad
    \frac{1-2s}{b}<\frac{1}{b+2}.
  \]
  Together with $1/(b+2)>0$, these inequalities show that the interval in \eqref{eq:superexp-lyapunov-weight-range} is non-empty. The last lower bound on $r$ gives $1-br-2s<0$.

  Using $N_{n+1}=N_n^b$, we have $a_n=N_n^{1-br}$ and $c_n=N_n^{1-r}$. Hence
  \begin{align}
    \frac{c_n^2}{a_na_{n+1}}
    &=N_n^{2-2r-(1-br)-(b-b^2r)}\notag\\
    &=N_n^{(b-1)((b+2)r-1)}\longrightarrow0,
    \label{eq:superexp-outflow-ratio}
  \end{align}
  since $r<1/(b+2)$. Choosing $n_*\geq1$ sufficiently large gives \eqref{eq:superexp-outflow-weight-bound}.

  The other two ratios are
  \begin{equation}
    \frac{d_n^2}{a_n}=N_n^{4\alpha-1+(b-2)r},
    \qquad
    \frac{w_n^2}{a_n}=N_n^{-1+(b-2)r}.
    \label{eq:superexp-summable-weight-ratios}
  \end{equation}
  The first exponent is negative because $r>(4\alpha-1)/(2-b)$. The second is negative because $b<2$ and $r>0$. Summability of negative powers of $N_n$, as recorded in \eqref{eq:superexp-summability}, proves convergence of both series in \eqref{eq:superexp-lyapunov-constants}. Their terms are positive, so both constants are strictly positive.
\end{proof}

The upper bound on $r$ makes the nonlinear outflow absorbable at high shells in \eqref{eq:superexp-outflow-young}. The lower bound involving $\alpha$ makes the viscous remainder summable in \eqref{eq:superexp-viscous-young}, while the lower bound involving $s$ ensures uniform convergence of the quadratic functional under the assumed $H_N^s$ bound in \eqref{eq:superexp-quadratic-tail}.

\subsection{The Lyapunov functional and the Riccati inequality}
\label{subsec:superexp-riccati}

Fix $s>\alpha_c$, and choose $r$ and $n_*$ as in Lemma~\ref{lem:superexp-lyapunov-weights}. With the weights \eqref{eq:superexp-lyapunov-weights}, set
\begin{equation}
  H(t):=\sum_{n=n_*}^{\infty}w_nu_n(t),
  \qquad
  P(t):=\sum_{n=n_*}^{\infty}a_nu_n(t)^2.
  \label{eq:superexp-lyapunov-functionals}
\end{equation}
In the next two lemmas, the solution is assumed to be non-negative and of finite energy, with a uniform $H_N^s$ bound on the interval under consideration. No additional smoothness assumption is used.

\begin{lemma}[Weighted identity and absorption estimates]
  \label{lem:superexp-weighted-identity}
  Let $u$ be a non-negative finite-energy solution on $[0,T]$ such that
  \[
    C_s:=\sup_{0\leq t\leq T}\|u(t)\|_{H_N^s}<\infty.
  \]
  The series defining $H$ and $P$ converge uniformly on $[0,T]$, and $H\in C^1([0,T])$. Furthermore,
  \begin{equation}
    \begin{aligned}
      H'(t)={}&w_{n_*}N_{n_*-1}u_{n_*-1}(t)^2+P(t)\\
      &-\sum_{n=n_*}^{\infty}c_nu_n(t)u_{n+1}(t)
       -\nu\sum_{n=n_*}^{\infty}d_nu_n(t).
    \end{aligned}
    \label{eq:superexp-weighted-identity}
  \end{equation}
  The last two series converge uniformly and satisfy
  \begin{equation}
    \sum_{n=n_*}^{\infty}c_nu_nu_{n+1}\leq\frac12P(t),
    \qquad
    \nu\sum_{n=n_*}^{\infty}d_nu_n\leq\frac14P(t)+C_\nu.
    \label{eq:superexp-lyapunov-absorption}
  \end{equation}
  In particular,
  \begin{equation}
    H'(t)\geq\frac14P(t)-C_\nu,\qquad 0\leq t\leq T.
    \label{eq:superexp-quadratic-lower-bound}
  \end{equation}
\end{lemma}

\begin{proof}
  We first establish uniform convergence. Since $a_n=N_n^{1-br}$ and $1-br-2s<0$, for every $M\geq n_*$,
  \begin{align}
    \sup_{0\leq t\leq T}\sum_{n=M}^{\infty}a_nu_n(t)^2
    &\leq C_s^2N_M^{1-br-2s}\longrightarrow0,
    \label{eq:superexp-quadratic-tail}\\
    \sup_{0\leq t\leq T}\sum_{n=M}^{\infty}w_nu_n(t)
    &\leq C_s\left(\sum_{n=M}^{\infty}N_n^{-2(r+s)}\right)^{1/2}
      \longrightarrow0.
    \label{eq:superexp-linear-tail}
  \end{align}
  The second estimate is Cauchy--Schwarz, and the numerical series converges because $r+s>0$. Each summand is continuous in time. Therefore $H$ and $P$ are continuous and bounded on $[0,T]$.

  Young's inequality and \eqref{eq:superexp-outflow-weight-bound} give, for $n\geq n_*$,
  \begin{align}
    c_nu_nu_{n+1}
    &\leq\frac14a_nu_n^2+\frac{c_n^2}{a_n}u_{n+1}^2
      \leq\frac14a_nu_n^2+\frac14a_{n+1}u_{n+1}^2,
    \label{eq:superexp-outflow-young}\\
    \nu d_nu_n
    &\leq\frac14a_nu_n^2+\nu^2\frac{d_n^2}{a_n}.
    \label{eq:superexp-viscous-young}
  \end{align}
  Summing proves \eqref{eq:superexp-lyapunov-absorption}. These inequalities also control the tails. For $M\geq n_*$, the nonlinear outflow tail is at most
  $\frac12\sum_{n=M}^{\infty}a_nu_n^2$, and the viscous tail is at most
  \[
    \frac14\sum_{n=M}^{\infty}a_nu_n^2
    +\nu^2\sum_{n=M}^{\infty}\frac{d_n^2}{a_n}.
  \]
  By \eqref{eq:superexp-quadratic-tail} and the convergence of $C_\nu$, both tails tend to zero uniformly on $[0,T]$. Thus the two series on the second line of \eqref{eq:superexp-weighted-identity} converge uniformly to continuous functions.

  To justify differentiation, define
  \[
    H_M(t):=\sum_{n=n_*}^Mw_nu_n(t).
  \]
  The coordinate equations and $N_nw_{n+1}=a_n$ imply
  \begin{align*}
    H_M'
    &=-\nu\sum_{n=n_*}^Md_nu_n
      +\sum_{n=n_*}^Mw_nN_{n-1}u_{n-1}^2
      -\sum_{n=n_*}^Mc_nu_nu_{n+1}\\
    &=w_{n_*}N_{n_*-1}u_{n_*-1}^2
      +\sum_{n=n_*}^{M-1}a_nu_n^2
      -\sum_{n=n_*}^Mc_nu_nu_{n+1}
      -\nu\sum_{n=n_*}^Md_nu_n.
  \end{align*}
  Here the second line separates the incoming boundary term and shifts the remaining incoming sum. The preceding tail estimates show that $H_M'$ converges uniformly to the continuous right-hand side of \eqref{eq:superexp-weighted-identity}. Since $H_M$ also converges uniformly, passage to the limit in
  \[
    H_M(t)-H_M(0)=\int_0^tH_M'(\tau)\,\dd\tau
  \]
  proves that $H\in C^1([0,T])$ and establishes \eqref{eq:superexp-weighted-identity}. The incoming boundary term is non-negative. Discarding it and applying \eqref{eq:superexp-lyapunov-absorption} yields
  \[
    H'\geq P-\frac12P-\frac14P-C_\nu
      =\frac14P-C_\nu.\qedhere
  \]
\end{proof}

\begin{lemma}[Riccati lower bound]
  \label{lem:superexp-riccati}
  Under the assumptions of Lemma~\ref{lem:superexp-weighted-identity},
  \begin{equation}
    H'(t)\geq\kappa H(t)^2-C_\nu,
    \qquad \kappa:=\frac{1}{4S_r}>0.
    \label{eq:superexp-riccati}
  \end{equation}
\end{lemma}

\begin{proof}
  Weighted Cauchy--Schwarz gives
  \[
    H(t)^2
    =\left(\sum_{n=n_*}^{\infty}\sqrt{a_n}\,u_n(t)
            \frac{w_n}{\sqrt{a_n}}\right)^2
    \leq P(t)S_r.
  \]
  Hence $P(t)\geq H(t)^2/S_r$, and \eqref{eq:superexp-quadratic-lower-bound} proves the assertion.
\end{proof}

\subsection{Proof of the blow-up theorem}
\label{subsec:superexp-blowup-proof}

\begin{proof}[Proof of Theorem~\ref{thm:superexp-blowup}]
  Fix $s>\alpha_c$, and choose $r,n_*,C_\nu,S_r$ as in Lemma~\ref{lem:superexp-lyapunov-weights}. Set $\kappa:=1/(4S_r)$. Let $e_{n_*}$ denote the sequence whose $n_*$th coordinate is one and whose other coordinates vanish. Choose $K>0$ large enough that the initial datum
  \begin{equation}
    h:=K e_{n_*}
    \qquad\text{satisfies}\qquad
    H(0)^2=w_{n_*}^2K^2>\frac{2C_\nu}{\kappa},
    \label{eq:superexp-blowup-datum}
  \end{equation}
  and define
  \begin{equation}
    T_R:=\frac{2}{\kappa H(0)}<\infty.
    \label{eq:superexp-riccati-time}
  \end{equation}
  These choices depend only on $s$ and the fixed parameters, not on a choice of energy solution.

  Proposition~\ref{prop:superexp-energy} supplies a global non-negative finite-energy solution with initial datum $h$. Let $u$ be any such solution. Suppose, for contradiction, that
  \begin{equation}
    \sup_{0\leq t\leq T_R}\|u(t)\|_{H_N^s}<\infty.
    \label{eq:superexp-blowup-contradiction-assumption}
  \end{equation}
  Lemma~\ref{lem:superexp-riccati} applies on $[0,T_R]$. Whenever $H(t)\geq H(0)$, the choice \eqref{eq:superexp-blowup-datum} yields
  \[
    H'(t)\geq\kappa H(t)^2-C_\nu
      \geq\frac{\kappa}{2}H(t)^2>0.
  \]
  In particular, $H'(0)>0$. A first return to the level $H(0)$ from above would have non-positive left derivative, contradicting this inequality. Thus $H(t)\geq H(0)>0$ throughout $[0,T_R]$, and
  \[
    \frac{\dd}{\dd t}\frac{1}{H(t)}
    =-\frac{H'(t)}{H(t)^2}\leq-\frac{\kappa}{2}.
  \]
  Integration gives
  \begin{equation}
    \frac{1}{H(t)}\leq\frac{1}{H(0)}-\frac{\kappa t}{2},
    \qquad 0\leq t\leq T_R.
    \label{eq:superexp-reciprocal-bound}
  \end{equation}
  For $t<T_R$, the right-hand side is positive, so
  \[
    H(t)\geq
      \left(\frac{1}{H(0)}-\frac{\kappa t}{2}\right)^{-1}
      \longrightarrow\infty
      \qquad\text{as }t\uparrow T_R.
  \]

  On the other hand, $r>0$ and \eqref{eq:superexp-summability} imply $w\in\ell^2$. Every non-negative finite-energy solution satisfies the energy inequality by Proposition~\ref{prop:superexp-energy-balance}. Since $\|h\|_{\ell^2}=K$, Cauchy--Schwarz gives
  \begin{equation}
    \begin{aligned}
      H(t)
      &\leq\left(\sum_{n=n_*}^{\infty}w_n^2\right)^{1/2}
          \|u(t)\|_{\ell^2}\\
      &\leq K\left(\sum_{n=n_*}^{\infty}w_n^2\right)^{1/2}<\infty,
      \qquad t\geq0.
    \end{aligned}
    \label{eq:superexp-linear-energy-bound}
  \end{equation}
  This uniform bound contradicts the preceding lower bound. Therefore \eqref{eq:superexp-blowup-contradiction-assumption} is false, proving \eqref{eq:superexp-loss-of-regularity} for every global non-negative finite-energy solution with initial datum $h$.

  Finally, $h$ is non-negative and finitely supported, hence smooth. Proposition~\ref{prop:superexp-continuation} supplies a local smooth solution and identifies it with every non-negative Galerkin solution on its smooth existence interval. If $T_{\max}>T_R$, a Galerkin solution would be smooth on $[0,T_R]$ and would therefore have a bounded $H_N^s$ norm there, contradicting \eqref{eq:superexp-loss-of-regularity}. Hence $T_{\max}\leq T_R<\infty$.
\end{proof}

\begin{remark}
  The conclusion is loss of weighted regularity, while finite-energy solutions continue to exist globally. The functional $H$ itself remains bounded by \eqref{eq:superexp-linear-energy-bound}. Its divergent Riccati lower bound was derived under the assumption of a uniform $H_N^s$ bound, and the contradiction shows that this assumption must fail. 
\end{remark}

Theorem~\ref{thm:superexp-blowup} establishes the blow-up side of the threshold. The following section proves global regularity at and above $\alpha_c$, completing the classification for super-exponential shells.

%% file: sections/5-superexp-regularity.tex
\section{Global Regularity for Super-Exponential Shells}
\label{sec:superexp-regularity}

We now prove global regularity at and above the exponent \eqref{eq:superexp-critical-exponent}. A fixed critical rescaling puts the equations into the form of the backward propagation lemma from Section~\ref{sec:acceleration}. The acceleration estimate makes the error uniformly small at high shells, and backward propagation gives uniform decay of the critical tail. A sequence of strict barriers then converts this decay into bounds in every weighted space. The argument includes the critical endpoint directly.

\subsection{The regularity theorem}
\label{subsec:superexp-regularity-statement}

Throughout this section, let
\[
  N_n=N_0^{b^n},\qquad N_0>1,\quad 1<b<2,\quad \nu>0,
  \qquad \alpha\geq\alpha_c=\frac1{b+2}.
\]

\begin{theorem}[Global regularity for super-exponential shells]
  \label{thm:superexp-regularity}
  Every non-negative $h\in H_N^\infty$ generates a global smooth solution of \eqref{eq:dyadic-model}. This solution is unique among global non-negative finite-energy solutions with initial datum $h$. In particular, for every $s>0$ and every finite $T>0$,
  \begin{equation}
    \sup_{0\leq t\leq T}\|u(t)\|_{H_N^s}<\infty.
    \label{eq:superexp-global-smoothness}
  \end{equation}
\end{theorem}

The idea of the proof is as follows. Let $[0,T_{\max})$ be the maximal smooth existence interval. If $T_{\max}<\infty$, then the solution loses regularity in some $H_N^s$ as $t\uparrow T_{\max}$. We then show that all $H_N^s$ norms are uniformly bounded on $[T_0,T_{\max})$. This contradicts the continuation criterion \eqref{eq:superexp-continuation}. To prove the uniform boundedness, we use the following two steps:
\begin{description}
  \item[Step 1] By Corollary~\ref{cor:superexp-local-smooth}, there exists a local smooth solution on $[0,T_0]$ for some $T_0>0$. 
  \item[Step 2] By Proposition~\ref{prop:superexp-short-wait-smoothing}, for every $T'\in(0,T_{\max})$ and every $s>0$, the solution satisfies
  \[
    \sup_{T'\leq t<T_{\max}}\|u(t)\|_{H_N^s}<\infty.
  \]
\end{description}

Fix such an initial datum, and choose a global non-negative Galerkin solution $u$ as in Proposition~\ref{prop:superexp-energy}. Proposition~\ref{prop:superexp-continuation} identifies it locally with the smooth solution and supplies its maximal smooth existence interval $[0,T_{\max})$. All estimates up to the proof of Theorem~\ref{thm:superexp-regularity} are made on this interval.

Set
\begin{equation}
  q:=\alpha_c,\qquad
  v_n:=N_n^qu_n,\qquad
  \lambda_n:=N_n^{2q},\qquad v_{-1}:=0.
  \label{eq:superexp-critical-rescaling}
\end{equation}
The exponent $q$ is fixed at the threshold, independently of the value of $\alpha\geq q$. Define also
\begin{equation}
  d_n:=\nu N_n^{2(\alpha-q)},\qquad
  R_n(t):=N_n^{-q}\dot u_n(t).
  \label{eq:superexp-damping-defect}
\end{equation}
Since $q(b+2)=1$, we have
\[
  \frac{1-2q}{b}=q,\qquad 1-(b+1)q=q.
\]
Thus, for $n\geq1$,
\[
  N_{n-1}u_{n-1}^2=N_n^qv_{n-1}^2,
  \qquad
  N_nu_nu_{n+1}=N_n^qv_nv_{n+1}.
\]
Dividing the coordinate equation by $N_n^q$ gives the recurrence required for backward propagation lemma \ref{lem:backward-propagation}:
\begin{equation}
  v_{n-1}(t)^2
  =d_nv_n(t)+v_n(t)v_{n+1}(t)+R_n(t),\qquad n\geq0,
  \label{eq:superexp-rescaled-recurrence}
\end{equation}
where the first term is zero when $n=0$. Equivalently,
\begin{equation}
  v_n'
  =\lambda_n\bigl(v_{n-1}^2-v_nv_{n+1}-d_nv_n\bigr)
  =\lambda_nR_n.
  \label{eq:superexp-rescaled-dynamics}
\end{equation}
The assumption $\alpha\geq q$ implies
\begin{equation}
  d_n\geq\nu>0,\qquad n\geq0,
  \label{eq:superexp-damping-lower-bound}
\end{equation}
because $N_n>1$. This is the damping condition needed for the uniform argument.

\subsection{Uniform decay of the critical tail}
\label{subsec:superexp-critical-tail}

\begin{proposition}[Uniform critical-tail decay]
  \label{prop:superexp-critical-tail}
  On the maximal smooth existence interval,
  \begin{equation}
    \lim_{n\to\infty}\sup_{0\leq t<T_{\max}}v_n(t)=0.
    \label{eq:superexp-uniform-critical-tail}
  \end{equation}
\end{proposition}

\begin{proof}
  By Corollary~\ref{cor:pointwise-acceleration},
  \begin{equation}
    \sup_{0\leq t<T_{\max}}|R_n(t)|
    \leq\sqrt{\mathcal A(0)}\,2^{n/2}N_n^{-q}
    \longrightarrow0.
    \label{eq:superexp-uniform-defect-bound}
  \end{equation}
  The convergence follows from \eqref{eq:superexp-defect-decay}, since $q>0$. The energy inequality gives, for every fixed $k\geq0$,
  \[
    \sup_{0\leq t<T_{\max}}v_k(t)
    \leq N_k^q\|h\|_{\ell^2}<\infty.
  \]
  Apply Lemma~\ref{lem:backward-propagation} to \eqref{eq:superexp-rescaled-recurrence} for $n\geq1$, with
  \[
    x_n=v_n,\qquad A=B=1,\qquad d_*=\nu,
    \qquad I=[0,T_{\max}).
  \]
  Its damping hypothesis is \eqref{eq:superexp-damping-lower-bound}; its remaining hypotheses follow from the preceding estimates. The conclusion is \eqref{eq:superexp-uniform-critical-tail}.
\end{proof}

In particular, with
\begin{equation}
  z_n:=v_n/\nu,\qquad z_{-1}:=0,
  \label{eq:superexp-normalized-variable}
\end{equation}
for every $\theta\in(0,1)$ there is an integer $N_\theta\geq1$ such that
\begin{equation}
  0\leq z_n(t)\leq\theta,
  \qquad n\geq N_\theta,\quad 0\leq t<T_{\max}.
  \label{eq:superexp-small-normalized-tail}
\end{equation}
The index is independent of time on the entire smooth existence interval.

\subsection{Smoothing from a small critical tail}
\label{subsec:superexp-tail-smoothing}

From \eqref{eq:superexp-rescaled-dynamics}, non-negativity, and \eqref{eq:superexp-damping-lower-bound}, we obtain
\begin{align}
  z_n'
  &=\nu\lambda_n\bigl(z_{n-1}^2-z_nz_{n+1}\bigr)
    -d_n\lambda_nz_n\notag\\
  &\leq\nu\lambda_n\bigl(z_{n-1}^2-z_n\bigr),\qquad n\geq0.
  \label{eq:superexp-scalar-damping}
\end{align}
We use this scalar inequality to improve the small tail into a rapidly decaying barrier. Fix $\theta\in(0,1)$ and choose
\begin{equation}
  b<p<2.
  \label{eq:superexp-barrier-exponent}
\end{equation}
These parameters remain fixed as the starting index $N$ increases.

Next results quantify the smoothing mechanism. First, Lemma~\ref{lem:superexp-barrier-step} gives a one-shell improvement: while the normalized amplitudes remain uniformly below $\theta$, once shell $n-1$ has reached its barrier value, the scalar damping inequality forces shell $n$ below its (smaller) barrier value after a controlled waiting time $\Delta t$. Lemma~\ref{lem:superexp-accumulated-wait} shows that these shell-by-shell waiting times have a finite sum. Consequently, after an arbitrarily short positive wait, the barrier holds simultaneously on all shells above $N$. Proposition~\ref{prop:superexp-short-wait-smoothing} then converts this rapid tail decay into uniform bounds in every weighted space $H_N^s$.

\begin{lemma}
  \label{lem:superexp-barrier-step}
  For each integer $N\geq1$, define
  \begin{equation}
    r_N:=\theta,\qquad
    r_n:=\theta^{p^{n-N}}=r_{n-1}^p,\qquad n>N.
    \label{eq:superexp-tail-barrier}
  \end{equation}
  For every $n>N$,
  \begin{equation}
    0<r_{n-1}^2<r_n<\theta,\qquad
    r_n-r_{n-1}^2
    \geq\theta^{p^{n-N}}\bigl(1-\theta^{2/p-1}\bigr).
    \label{eq:superexp-strict-barrier-gap}
  \end{equation}
  The numerical waiting times
  \begin{equation}
    \Delta t_n:=\frac1{\nu\lambda_n}
    \log\left(\frac{\theta-r_{n-1}^2}{r_n-r_{n-1}^2}\right),
    \qquad n>N,
    \label{eq:superexp-shell-waiting-time}
  \end{equation}
  are positive and finite, and satisfy
  \begin{equation}
    \Delta t_n\leq
    \frac{C_{p,\theta}(1+p^{n-N})}{\nu N_n^{2q}},
    \label{eq:superexp-shell-waiting-bound}
  \end{equation}
  where $C_{p,\theta}$ is independent of $N$ and $n$.

  Suppose that $z_m(t)\leq\theta$ for every $m\geq N$ and $0\leq t<T_{\max}$. If, for some $n>N$ and $0\leq t_{n-1}<T_{\max}$,
  \[
    z_{n-1}(t)\leq r_{n-1},\qquad t_{n-1}\leq t<T_{\max},
  \]
  then, with $t_n:=t_{n-1}+\Delta t_n$,
  \begin{equation}
    z_n(t)\leq r_n\qquad\text{whenever }t_n\leq t<T_{\max}.
    \label{eq:superexp-one-shell-barrier}
  \end{equation}
\end{lemma}

\begin{proof}
  Write $a=p^{n-N}>1$. Since $p<2$,
  \[
    r_{n-1}^2=\theta^{(2/p)a}<\theta^a=r_n<\theta.
  \]
  Moreover,
  \[
    r_n-r_{n-1}^2
    =\theta^a\bigl(1-\theta^{(2/p-1)a}\bigr)
    \geq\theta^a\bigl(1-\theta^{2/p-1}\bigr).
  \]
  This proves \eqref{eq:superexp-strict-barrier-gap}. The numerator in \eqref{eq:superexp-shell-waiting-time} is strictly larger than its positive denominator, so $0<\Delta t_n<\infty$. Also,
  \begin{align*}
    \Delta t_n
    &\leq\frac1{\nu\lambda_n}
      \log\left(\frac{\theta^{1-a}}{1-\theta^{2/p-1}}\right)\\
    &=\frac{(a-1)|\log\theta|-\log(1-\theta^{2/p-1})}{\nu\lambda_n}\\
    &\leq\frac{C_{p,\theta}(1+a)}{\nu N_n^{2q}}.
  \end{align*}
  These numerical estimates require no assumption on the lifespan of the solution.

  For the damping step, \eqref{eq:superexp-scalar-damping} and the bound on $z_{n-1}$ give
  \[
    z_n'+\nu\lambda_nz_n\leq\nu\lambda_nr_{n-1}^2
    \qquad\text{on }[t_{n-1},T_{\max}).
  \]
  Integrating this inequality yields, for $t_{n-1}\leq t<T_{\max}$,
  \begin{align*}
    z_n(t)
    &\leq r_{n-1}^2+
      \bigl(z_n(t_{n-1})-r_{n-1}^2\bigr)
      \exp\bigl(-\nu\lambda_n(t-t_{n-1})\bigr)\\
    &\leq r_{n-1}^2+
      \bigl(\theta-r_{n-1}^2\bigr)
      \exp\bigl(-\nu\lambda_n(t-t_{n-1})\bigr).
  \end{align*}
  At a time $t\geq t_n$ within the existence interval, the last expression is at most
  \[
    r_{n-1}^2+\bigl(\theta-r_{n-1}^2\bigr)
      \exp(-\nu\lambda_n\Delta t_n)=r_n.
  \]
  This proves \eqref{eq:superexp-one-shell-barrier}.
\end{proof}

\begin{lemma}[Finite accumulated waiting time]
  \label{lem:superexp-accumulated-wait}
  For each $N\geq1$, let the barriers and numerical waiting times be those of Lemma~\ref{lem:superexp-barrier-step}, and set
  \begin{equation}
    t_N:=0,\qquad
    t_n:=\sum_{j=N+1}^n\Delta t_j\quad(n>N),\qquad
    T_N:=\sum_{n>N}\Delta t_n.
    \label{eq:superexp-accumulated-wait}
  \end{equation}
  Then $0<T_N<\infty$ and $T_N\to0$ as $N\to\infty$.

  If the solution satisfies $z_m(t)\leq\theta$ for all $m\geq N$ and $0\leq t<T_{\max}$, and if $T_N<T_{\max}$, then
  \begin{equation}
    z_n(t)\leq\theta^{p^{n-N}},
    \qquad n\geq N,\quad T_N\leq t<T_{\max}.
    \label{eq:superexp-barrier-after-wait}
  \end{equation}
\end{lemma}

\begin{proof}
  Although their dependence on $N$ is suppressed in the notation, the barriers and waiting times are defined separately for each starting index. By \eqref{eq:superexp-shell-waiting-bound},
  \begin{equation}
    T_N\leq\frac{C_{p,\theta}}{\nu}
       \sum_{n>N}(1+p^{n-N})N_n^{-2q}
    \leq\frac{C_{p,\theta}}{\nu}
       \sum_{n>N}(1+p^n)N_n^{-2q}.
    \label{eq:superexp-total-wait-bound}
  \end{equation}
  The last series converges. Indeed, since $b>1$ and $q>0$, for all sufficiently large $n$,
  \[
    (1+p^n)N_n^{-2q}
    \leq2\exp\bigl(n\log p-2qb^n\log N_0\bigr)
    \leq2N_n^{-q}.
  \]
  Summability follows from \eqref{eq:superexp-summability}. The bound in \eqref{eq:superexp-total-wait-bound} is a tail of a convergent series independent of $N$, so $T_N\to0$. Positivity of the waiting times gives $T_N>0$.

  Now assume the stated small-tail bound and $T_N<T_{\max}$. All the times $t_n$ then lie inside the smooth existence interval. The initial step is $z_N(t)\leq\theta=r_N$ for $t\geq t_N=0$. Induction using Lemma~\ref{lem:superexp-barrier-step} gives
  \[
    z_n(t)\leq r_n,\qquad t_n\leq t<T_{\max},\quad n\geq N.
  \]
  Since $t_n\leq T_N$, this implies \eqref{eq:superexp-barrier-after-wait}. The numerical estimate on $T_N$ was established before the induction, so no step evaluates the solution beyond its existence interval.
\end{proof}

\begin{proposition}[Smoothing after short wait]
  \label{prop:superexp-short-wait-smoothing}
  For every $\varepsilon>0$, there exists $T_0$ with
  $0<T_0<\min\{\varepsilon,T_{\max}\}$ such that, for every $s>0$,
  \begin{equation}
    \sup_{T_0\leq t<T_{\max}}\|u(t)\|_{H_N^s}<\infty.
    \label{eq:superexp-short-wait-smoothing}
  \end{equation}
  The same $T_0$ works for all $s>0$.
\end{proposition}

\begin{proof}
  Fix $\theta\in(0,1)$ and $p\in(b,2)$. Proposition~\ref{prop:superexp-critical-tail} gives an index $N_\theta$ as in \eqref{eq:superexp-small-normalized-tail}. For every $N\geq N_\theta$, the same small-tail bound holds from index $N$ onward. By Lemma~\ref{lem:superexp-accumulated-wait}, we may choose such an $N$ large enough that
  \[
    0<T_N<\min\{\varepsilon,T_{\max}\}.
  \]
  Set $T_0:=T_N$. Then \eqref{eq:superexp-barrier-after-wait} gives
  \[
    0\leq u_n(t)=\nu N_n^{-q}z_n(t)
    \leq\nu N_n^{-q}\theta^{p^{n-N}},
    \qquad n\geq N,\quad T_0\leq t<T_{\max}.
  \]
  For any $s>0$, the high-shell contribution is therefore bounded by
  \begin{equation}
    \sup_{T_0\leq t<T_{\max}}\sum_{n\geq N}N_n^{2s}u_n(t)^2
    \leq\nu^2\sum_{n\geq N}N_n^{2(s-q)}\theta^{2p^{n-N}}.
    \label{eq:superexp-barrier-sobolev-bound}
  \end{equation}
  The $n$th term in the numerical series is
  \[
    \exp\bigl(2(s-q)b^n\log N_0-2|\log\theta|p^{n-N}\bigr).
  \]
  Since $p>b$, for all sufficiently large $n$ its exponent is at most
  $-|\log\theta|p^{n-N}$. The resulting majorant is summable, so the series in \eqref{eq:superexp-barrier-sobolev-bound} converges. This argument also covers $s\leq q$, when the first term of the exponent is non-positive.

  For the finitely many remaining shells, the energy inequality gives
  \[
    \sup_{0\leq t<T_{\max}}\sum_{n=0}^{N-1}N_n^{2s}u_n(t)^2
    \leq\|h\|_{\ell^2}^2\sum_{n=0}^{N-1}N_n^{2s}<\infty.
  \]
  Combining the two bounds proves \eqref{eq:superexp-short-wait-smoothing}. The index $N$ and the time $T_0$ were chosen before fixing $s$, which proves the final assertion.
\end{proof}

\begin{remark}
  The two inequalities in $b<p<2$ have distinct roles. The inequality $p<2$ gives the strict gap $r_n-r_{n-1}^2>0$, which makes each damping step possible. The inequality $p>b$ makes the barrier decay faster than every power of the shell scales. Thus the restriction $b<2$ enters explicitly in this smoothing argument.
\end{remark}

\subsection{Global regularity and the critical endpoint}
\label{subsec:superexp-global-and-endpoint}

\begin{proof}[Proof of Theorem~\ref{thm:superexp-regularity}]
  Suppose, for contradiction, that $T_{\max}<\infty$. Choose $0<\varepsilon<T_{\max}/2$ in Proposition~\ref{prop:superexp-short-wait-smoothing}. It supplies $T_0<\varepsilon$ such that every $H_N^s$ norm is uniformly bounded on $[T_0,T_{\max})$. The solution is smooth on the compact interval $[0,T_0]$, so
  \[
    \sup_{0\leq t<T_{\max}}\|u(t)\|_{H_N^s}<\infty
    \qquad\text{for every }s>0.
  \]
  Fix $\sigma>1/(2-b)$. The embedding \eqref{eq:weighted-embeddings} implies
  \[
    \sup_{0\leq t<T_{\max}}\|u(t)\|_{X_N^\sigma}<\infty,
  \]
  contrary to the continuation criterion \eqref{eq:superexp-continuation}. Hence $T_{\max}=\infty$, and the solution is globally smooth in the sense of Definition~\ref{def:smoothness}. In particular, \eqref{eq:superexp-global-smoothness} holds on every finite time interval.

  If $\widetilde u$ is any global non-negative finite-energy solution with the same initial datum, Proposition~\ref{prop:superexp-weak-strong} applied on each $[0,T]$ gives $\widetilde u=u$ there. Since $T$ is arbitrary, uniqueness holds globally.
\end{proof}

We now explain how the argument includes the critical endpoint. When $\alpha=q=\alpha_c$, the damping coefficients are exactly $d_n=\nu$, and \eqref{eq:superexp-rescaled-dynamics} reduces to
\begin{equation}
  v_n'=\lambda_nQ_n,\qquad
  Q_n:=v_{n-1}^2-v_n(v_{n+1}+\nu),\qquad n\geq0,
  \label{eq:superexp-endpoint-defect}
\end{equation}
with $v_{-1}=0$. Here $Q_n=R_n=N_n^{-q}\dot u_n$, so
\begin{equation}
  \mathcal A(t)=\sum_{n\geq0}2^{-n}N_n^{2q}Q_n(t)^2,
  \qquad
  \sup_{0\leq t<T_{\max}}|Q_n(t)|
  \leq\sqrt{\mathcal A(0)}\,2^{n/2}N_n^{-q}\longrightarrow0.
  \label{eq:superexp-endpoint-defect-bound}
\end{equation}
The backward propagation lemma applies with $A=B=1$ and constant damping $\nu$. So Proposition~\ref{prop:superexp-critical-tail} remains valid at the endpoint. The normalized dynamics satisfy
\[
  z_n'=\nu\lambda_n\bigl(z_{n-1}^2-z_nz_{n+1}-z_n\bigr)
  \leq\nu\lambda_n\bigl(z_{n-1}^2-z_n\bigr),
\]
which is the same inequality used in the barrier construction. Thus the endpoint is proved directly, without taking a limit from $\alpha>\alpha_c$.

For $\alpha>\alpha_c$, the same rescaling gives $d_n=\nu N_n^{2(\alpha-q)}\geq\nu$. The argument uses only this lower bound, both in backward propagation and in the scalar damping inequality. This is why one proof covers the endpoint and all larger dissipation exponents.

\subsection{The sharp regularity threshold}
\label{subsec:superexp-sharp-threshold}

\begin{corollary}[Sharp regularity threshold]
  \label{cor:superexp-sharp-threshold}
  Fix $N_0>1$, $1<b<2$, and $\nu>0$. For the super-exponential model \eqref{eq:dyadic-model}, the following classification holds:
  \begin{enumerate}
    \item If $0\leq\alpha<1/(b+2)$, then, for every $s>1/(b+2)$, there is a non-negative finitely supported initial datum for which every global non-negative finite-energy solution loses $H_N^s$ regularity in finite time. The corresponding maximal smooth existence time is finite.
    \item If $\alpha\geq1/(b+2)$, every non-negative smooth initial datum generates a unique global non-negative finite-energy solution, and this solution is globally smooth.
  \end{enumerate}
  Consequently, $\alpha_c=1/(b+2)$ is the sharp threshold for global regularity of all non-negative smooth initial data, and the critical point belongs to the regular regime.
\end{corollary}

\begin{proof}
  The first assertion is Theorem~\ref{thm:superexp-blowup}; the second is Theorem~\ref{thm:superexp-regularity}. Below $\alpha_c$, the finite-time loss of regularity for suitable smooth initial data excludes global regularity for all such data. At and above $\alpha_c$, global regularity holds for every non-negative smooth datum. These two statements establish sharpness.
\end{proof}

The next section applies the same acceleration and backward propagation estimates to geometric shells. There the conversion of the time-derivative bound into a vanishing rescaled error depends on the shell ratio, while passing from critical-tail decay to smoothness uses the classical criterion from Section~\ref{sec:preliminaries}.

%% file: sections/6-classical-regularity.tex
\section{Global Regularity for Geometric Shells}
\label{sec:classical-regularity}

We apply the acceleration estimate and backward propagation principle to the classical model. The rescaled equation again has a uniformly vanishing error under an explicit condition on the shell ratio. The resulting critical-tail decay implies smoothness by Proposition~\ref{prop:classical-tail-criterion}. We then compare the roles of the shell geometry in the two regularity arguments.

\subsection{The regularity theorem}
\label{subsec:classical-regularity-statement}

Throughout this section, we consider \eqref{eq:dyadic-model} with
\[
  N_n=N_0\Lambda^n,\qquad N_0>0,\quad\Lambda>1,\quad\nu>0,\quad\alpha>0.
\]
Recall from \eqref{eq:geometric-acceleration-loss} that
\[
  \rho:=\frac{\log2}{2\log\Lambda}.
\]

\begin{theorem}[Global regularity for geometric shells]
  \label{thm:classical-regularity}
  Suppose that
  \begin{equation}
    \alpha\geq\frac13,\qquad 4\alpha-1>\rho.
    \label{eq:classical-regularity-conditions}
  \end{equation}
  For every non-negative $h\in H_N^\infty$, the global Galerkin solution of \eqref{eq:dyadic-model} is smooth on $[0,\infty)$ and is unique among global non-negative Leray--Hopf solutions with initial datum $h$. In particular, for every $s>0$ and every finite $T>0$,
  \begin{equation}
    \sup_{0\leq t\leq T}\|u(t)\|_{H_N^s}<\infty.
    \label{eq:classical-global-smoothness}
  \end{equation}
\end{theorem}

\begin{corollary}[Regularity down to the critical exponent]
  \label{cor:classical-large-ratio}
  If $\Lambda>2^{3/2}$, then the conclusions of Theorem~\ref{thm:classical-regularity} hold for every $\alpha\geq1/3$.
\end{corollary}

\begin{proof}
  The condition on $\Lambda$ gives $\rho<1/3$. Thus $4\alpha-1\geq1/3>\rho$ whenever $\alpha\geq1/3$.
\end{proof}

The more precise condition \eqref{eq:classical-regularity-conditions} will be retained in the proof. For a fixed $\alpha\geq1/3$, its shell-ratio requirement is
\[
  \Lambda>2^{1/(2(4\alpha-1))},
\]
which allows smaller ratios as $\alpha$ increases. This is a sufficient condition supplied by the present argument.

Fix a non-negative smooth initial datum $h$, and choose a global non-negative Galerkin solution using Proposition~\ref{prop:classical-energy}. Let $[0,T_{\max})$ be its maximal smooth existence interval from Proposition~\ref{prop:classical-local}. Recall the classical critical weight $q_{\mathrm{cl}}(\alpha)=1-2\alpha$ from \eqref{eq:classical-critical-weight}, and set
\begin{equation}
  x_n(t):=N_n^{1-2\alpha}u_n(t),\qquad
  R_n(t):=N_n^{1-4\alpha}\dot u_n(t),\qquad n\geq0.
  \label{eq:classical-rescaled-variables}
\end{equation}
Multiplication of \eqref{eq:dyadic-model} by $N_n^{1-4\alpha}$ gives, for $n\geq1$,
\begin{equation}
  A_\alpha x_{n-1}(t)^2
  =\nu x_n(t)+B_\alpha x_n(t)x_{n+1}(t)+R_n(t),
  \label{eq:classical-rescaled-recurrence}
\end{equation}
where
\begin{equation}
  A_\alpha:=\Lambda^{1-4\alpha},\qquad
  B_\alpha:=\Lambda^{2\alpha-1}.
  \label{eq:classical-recurrence-coefficients}
\end{equation}
Indeed, the incoming coefficient becomes
$N_n^{1-4\alpha}N_{n-1}^{4\alpha-1}=\Lambda^{1-4\alpha}$,
and the outgoing coefficient becomes
$N_n^{1-2\alpha}N_{n+1}^{2\alpha-1}=\Lambda^{2\alpha-1}$.
For the first active shell, the same calculation gives
\[
  0=\nu x_0+B_\alpha x_0x_1+R_0.
\]
Only the recurrence for $n\geq1$ is needed in the backward propagation lemma \ref{lem:backward-propagation}, whose hypotheses require a bound on $x_0$ rather than a prescribed value there.

The coefficient condition is
\begin{equation}
  \frac{B_\alpha}{A_\alpha}=\Lambda^{6\alpha-2}\geq1
  \quad\Longleftrightarrow\quad\alpha\geq\frac13.
  \label{eq:classical-coefficient-order}
\end{equation}
This accounts for the first hypothesis in \eqref{eq:classical-regularity-conditions}. The second hypothesis ensures uniform decay of $R_n$.

\subsection{Uniform decay of the classical critical tail}
\label{subsec:classical-uniform-tail}

\begin{proposition}[Uniform classical critical-tail decay]
  \label{prop:classical-uniform-tail}
  Under \eqref{eq:classical-regularity-conditions},
  \begin{equation}
    \lim_{n\to\infty}\sup_{0\leq t<T_{\max}}
      N_n^{1-2\alpha}u_n(t)=0.
    \label{eq:classical-uniform-tail}
  \end{equation}
\end{proposition}

\begin{proof}
  Corollary~\ref{cor:pointwise-acceleration} and the identity
  $2^{n/2}=N_0^{-\rho}N_n^\rho$ give
  \begin{align}
    \sup_{0\leq t<T_{\max}}|R_n(t)|
    &\leq\sqrt{\mathcal A(0)}\,2^{n/2}N_n^{1-4\alpha}\notag\\
    &=N_0^{-\rho}\sqrt{\mathcal A(0)}
      N_n^{-(4\alpha-1-\rho)}\longrightarrow0.
    \label{eq:classical-uniform-defect-bound}
  \end{align}
  The convergence follows from $4\alpha-1>\rho$ and is uniform throughout the smooth existence interval. On the other hand, the energy inequality gives, for every fixed $k\geq0$,
  \[
    \sup_{0\leq t<T_{\max}}x_k(t)
    \leq N_k^{1-2\alpha}\|h\|_{\ell^2}<\infty.
  \]
  Apply Lemma~\ref{lem:backward-propagation} to \eqref{eq:classical-rescaled-recurrence} with
  \[
    A=A_\alpha,\qquad B=B_\alpha,\qquad d_n=d_*=\nu,
    \qquad I=[0,T_{\max}).
  \]
  The condition $B\geq A$ follows from \eqref{eq:classical-coefficient-order}; all other hypotheses in backward propagation are satisfied. Then we obtain \eqref{eq:classical-uniform-tail}.
\end{proof}

The two conditions in \eqref{eq:classical-regularity-conditions} enter at separate points: $\alpha\geq1/3$ orders the quadratic coefficients $A_\alpha$ and $B_\alpha$, while $4\alpha-1>\rho$ makes the rescaled time derivative a vanishing error. Once these are established, no further restriction is used in the tail argument.

\subsection{Global regularity, the endpoint, and sharpness}
\label{subsec:classical-endpoint-sharpness}

\begin{proof}[Proof of Theorem~\ref{thm:classical-regularity}]
  Suppose that $T_{\max}<\infty$. The chosen Galerkin solution is a global finite-energy solution, so each coordinate is continuous at $T_{\max}$. For every $n\geq0$,
  \[
    N_n^{1-2\alpha}u_n(T_{\max})
    \leq\sup_{0\leq t<T_{\max}}N_n^{1-2\alpha}u_n(t).
  \]
  Proposition~\ref{prop:classical-uniform-tail} therefore gives
  \begin{equation}
    \lim_{n\to\infty}\sup_{0\leq t\leq T_{\max}}
      N_n^{1-2\alpha}u_n(t)=0.
    \label{eq:classical-closed-interval-tail}
  \end{equation}
  Apply Proposition~\ref{prop:classical-tail-criterion} on $[0,T_{\max}]$. It follows that $u$ is smooth on this closed interval. In particular, for any fixed $r>\max\{0,1-2\alpha\}$,
  \[
    \sup_{0\leq t\leq T_{\max}}\|u(t)\|_{X_N^r}<\infty,
  \]
  contradicting the continuation criterion \eqref{eq:classical-continuation}. Thus $T_{\max}=\infty$, and \eqref{eq:classical-global-smoothness} follows from the definition of smoothness. Proposition~\ref{prop:classical-local} gives uniqueness against a non-negative Leray--Hopf solution on every finite time interval, hence on $[0,\infty)$.
\end{proof}

Above argument holds for $\alpha=1/3$. Because at $\alpha=1/3$, the rescaling and coefficients are
\begin{equation}
  q_{\mathrm{cl}}(1/3)=\frac13,\qquad
  A_{1/3}=B_{1/3}=\Lambda^{-1/3},\qquad
  R_n=N_n^{-1/3}\dot u_n.
  \label{eq:classical-endpoint-coefficients}
\end{equation}
The equality of the quadratic coefficients is allowed in Lemma~\ref{lem:backward-propagation}. The error estimate, however, requires
\[
  \frac13>\rho
  \quad\Longleftrightarrow\quad\Lambda>2^{3/2}.
\]
Thus the proof includes the critical dissipation exponent when $\Lambda>2^{3/2}$. The case $\alpha=1/3$ with $1<\Lambda\leq2^{3/2}$ is not covered by the present argument.

\begin{remark}
  The equality $\alpha=1/3$ and the equality $\Lambda=2^{3/2}$ play different roles. The former gives equal quadratic coefficients and causes no loss in backward propagation. At $\alpha=1/3$ and $\Lambda=2^{3/2}$, the exponent in \eqref{eq:classical-uniform-defect-bound} is zero, so that estimate alone does not imply decay of the defect. No assertion about regularity at this pair of parameters follows from the present argument. We do not claim that the shell-ratio restriction is optimal.
\end{remark}

\begin{corollary}[Sharp threshold for geometric shells]
  \label{cor:classical-sharp-threshold}
  Fix $N_0>0$, $\nu>0$, and $\Lambda>2^{3/2}$. Within the parameter range $\alpha>0$, the exponent $\alpha=1/3$ is the sharp threshold for global regularity of all non-negative smooth initial data. For $\alpha\geq1/3$, every such datum produces a globally smooth solution, unique among global non-negative Leray--Hopf solutions. For every $0<\alpha<1/3$, there are non-negative smooth initial data whose corresponding maximal smooth solution has finite existence time.
\end{corollary}

\begin{proof}
  The regularity assertion follows from Corollary~\ref{cor:classical-large-ratio}. The blow-up assertion is the consequence of \cite[Theorem~5.3]{Cheskidov2008} in our normalization. To make this conversion explicit, set
  \[
    U_k:=\frac{N_0}{\Lambda^2}u_{k-1}\quad(k\geq1),\qquad
    U_0:=0,\qquad
    \widehat\nu:=\nu\left(\frac{N_0}{\Lambda}\right)^{2\alpha}.
  \]
  Then the equations become
  \begin{equation}
    \dot U_k+\widehat\nu\Lambda^{2\alpha k}U_k
    =\Lambda^kU_{k-1}^2-\Lambda^{k+1}U_kU_{k+1},
    \qquad k\geq1,
    \label{eq:cheskidov-normalization}
  \end{equation}
  with zero forcing. This is the normalization of that theorem. The change of variables preserves non-negativity, finite support, and smoothness; for each fixed weight, the corresponding norms differ only by a positive constant.

  For $0<\alpha<1/3$, the cited theorem gives failure of local time integrability of a higher weighted norm cubed when the initial weighted norm is sufficiently large. Its size condition can be met by increasing the amplitude of a non-negative finitely supported datum. A globally smooth solution would have every weighted norm bounded on each compact time interval, contradicting this conclusion. The local theory in Proposition~\ref{prop:classical-local} therefore gives a finite maximal smooth existence time for such data. The sharpness statement combines this existing blow-up result with the regularity proved here.
\end{proof}

\subsection{Comparison of the two shell settings}
\label{subsec:comparison-shell-configurations}

The method originated in the study of super-exponentially separated shells. The preceding proofs show how its two main ingredients, the weighted acceleration identity and backward propagation, apply to both configurations. The acceleration cancellation depends on the relation $2w_{n+1}=w_n$, rather than on the spacing of the shells. After rescaling, each model gives a recurrence to which Lemma~\ref{lem:backward-propagation} applies once the error tends to zero. In both cases, the energy inequality supplies the required bounds at fixed indices.

The rescalings organize the coefficients differently. For super-exponential shells, the fixed weight $q=1/(b+2)$ gives
\[
  v_{n-1}^2
  =\nu N_n^{2(\alpha-q)}v_n+v_nv_{n+1}+N_n^{-q}\dot u_n.
\]
The quadratic coefficients are equal, and increasing $\alpha$ strengthens the linear damping. For geometric shells, the choice of $q_{\mathrm{cl}}(\alpha)=1-2\alpha$ keeps the linear damping constant, while the quadratic coefficients satisfy
$B_\alpha/A_\alpha=\Lambda^{6\alpha-2}$. Thus the regularity condition appears as a damping lower bound in the first rescaling and as an ordering of the quadratic coefficients in the second.

The main difference in error control is expressed by \eqref{eq:superexp-defect-decay} and \eqref{eq:geometric-defect-decay}. For super-exponential shells, every positive power of $N_n$ dominates $2^{n/2}$, so the defect $N_n^{-q}\dot u_n$ tends to zero uniformly. For geometric shells, the same acceleration bound loses the power $\rho$ of $N_n$, and the defect $N_n^{1-4\alpha}\dot u_n$ is controlled by a decaying bound only when $4\alpha-1>\rho$. This is the additional condition encountered when transferring the method to geometric shells.

The two threshold values also have a common interpretation through the energy flux. For the formal profile $u_n=N_n^{-\eta}$, the flux through shell $n$ is
\begin{equation}
  N_nu_n^2u_{n+1}
  =
  \begin{cases}
    N_n^{1-(b+2)\eta}, & N_{n+1}=N_n^b,\\
    \Lambda^{-\eta}N_n^{1-3\eta}, & N_{n+1}=\Lambda N_n.
  \end{cases}
  \label{eq:comparison-flux-scaling}
\end{equation}
The powers become zero at $\eta=1/(b+2)$ and $\eta=1/3$, respectively. This calculation explains the similarity between the threshold scales; the regularity and sharpness assertions themselves follow from the preceding proofs and the cited classical blow-up theorem. These formal profiles are heuristic and not being asserted to solve the viscous equations.

Although $1/(b+2)\to1/3$ as $b\downarrow1$, the current super-exponential family does not directly converge to geometric shells: for fixed $N_0$ and each fixed $n$, one has $N_0^{b^n}\to N_0$ as $b\downarrow 1$. The correspondence established here concerns the proof mechanism and the critical flux scales, rather than a convergence theorem between the two models.

Finally, non-negativity has a common role in both arguments. It gives the favorable sign of the energy flux and of the nonlinear term in the acceleration dissipation. The geometric shell-ratio condition enters the defect estimate, while the super-exponential restriction $b<2$ enters the local construction and the barrier smoothing step. These are the respective parameter restrictions of the results proved here.

%% file: sections/7-acknowledgments.tex
\section{Acknowledgments}
\label{sec:acknowledgments}

The author is grateful to Professor Cheskidov for his continued guidance and encouragement during the preparation of this manuscript. The author also thanks him for regularly reading earlier versions of the notes and providing valuable suggestions on their presentation.